\documentclass[10pt]{article}
\usepackage{verbatim}
\usepackage[multiple]{footmisc}
\usepackage{amsfonts,color}
\usepackage{longtable}
\usepackage{makecell, multirow, tabularx}
\usepackage{amsmath,amssymb}
\usepackage{adjustbox}
\usepackage{epsfig}
\usepackage{lscape}
\usepackage{amssymb}
\usepackage{graphicx}
\usepackage[utf8]{inputenc}	
\usepackage{amsthm,amssymb}
\usepackage{aligned-overset}
\usepackage[english]{babel}
\usepackage{dsfont}
\usepackage{bbm}
\usepackage[T1]{fontenc}
\usepackage{colonequals}
\usepackage{geometry}
\usepackage{mathtools}
\usepackage{csquotes}
\usepackage{url}
\usepackage{lipsum}
\usepackage{array}
\usepackage{tablefootnote}
\usepackage{wrapfig}
\usepackage[utf8]{inputenc}
\usepackage[dvipsnames]{xcolor}
\usepackage{accents}
\usepackage{marvosym}
\usepackage{centernot}
\usepackage{tikz}
\usepackage{titling}
\usepackage{multicol}

\usepackage{hyperref}

\usepackage[bibtex=true]{agn_bib}

\newtheorem{theorem}{Theorem}[section]
\newtheorem{lemma}[theorem]{Lemma}

\theoremstyle{definition}

\newtheorem{remark}[theorem]{Remark}

\renewcommand{\phi}{\varphi}
\allowdisplaybreaks[1]
\makeindex
\theoremstyle{definition}

\renewcommand{\rho}{\varrho}

\newcommand{\OO}{{\rm O}}

\newcommand{\R}{\mathbb{R}}

\newcommand{\tr}{\mathrm{tr}}

\newcommand{\GL}{\mathrm{GL}}

\newcommand{\sym}{\mathrm{sym}}
\newcommand{\skw}{\mathrm{skew}}

\newcommand{\id}{{\boldsymbol{\mathbbm{1}}}}
\newcommand*\dif{\mathop{}\!\mathrm{d}}
\newcommand{\norm}[1]{\lVert #1 \rVert}
\def\dvg{\textnormal{Div}}
\def\H{\mathbb{H}}
\def\C{\mathbb{C}}
\def\dd{\displaystyle}
\def\sk{\textnormal{skew}}
\DeclareMathOperator{\ZJ}{ZJ}
\DeclareMathOperator{\diag}{diag}
\DeclareMathOperator{\Sym}{Sym}
\DeclareMathOperator{\bfsym}{\textbf{sym}}
\DeclareMathOperator{\iso}{iso}

\DeclareMathOperator{\Cof}{Cof}

\DeclareMathOperator{\GN}{GN}

\DeclareMathOperator{\TR}{TR}
\DeclareMathOperator{\CR}{CR}
\DeclareMathOperator{\Old}{Old}
\DeclareMathOperator{\dev}{dev}
\newcommand{\DD}{\mathrm{D}}
\newcommand{\WW}{\mathrm{W}}
\DeclareMathOperator{\Biot}{Biot}
\newlength{\dhatheight}
\newcommand{\doublehat}[1]{%
	\settoheight{\dhatheight}{\ensuremath{\widehat{#1}}}%
	\addtolength{\dhatheight}{-0.35ex}%
	\widehat{\vphantom{\rule{1pt}{\dhatheight}}%
		\smash{\widehat{#1}}}}
\newcommand{\notiff}{%
	\mathrel{{\ooalign{\hidewidth$\not\phantom{"}$\hidewidth\cr$\iff$}}}}

\thanksmarkseries{arabic}

\providecommand{\mathbbs}[1]{{\scriptstyle{\mathbb{#1}}}}

\providecommand{\otimesdown}    {\mspace{4mu} \underline{\otimes} \mspace{4mu}}
\providecommand{\otimesup}      {\mspace{4mu} \overline{\otimes}  \mspace{4mu}}
\providecommand{\otimesdownup}  {\mspace{4mu} \overline{\underline{\otimes}} \mspace{4mu}}

\providecommand{\dc}{:}

\providecommand{\ldot}{\mspace{1mu} . \mspace{1mu}}

\usepackage{datetime2}

\definecolor{green}{rgb}{0.0, 0.5, 0.0}

\title{Major symmetry of the induced tangent stiffness tensor for the Zaremba-Jaumann rate and Kirchhoff stress in hyperelasticity: two different approaches}
\author{
Salvatore Federico\thanks{
Salvatore Federico, The University of Calgary, Department of Mechanical and Manufacturing Engineering, email: salvatore.federico@ucalcary.ca (corresponding author)
}, \qquad
Sebastian Holthausen\thanks{
Sebastian Holthausen, University of Duisburg-Essen, Chair for Nonlinear Analysis and Modelling,  Faculty of Mathematics, Thea-Leymann-Stra{\ss}e 9,
D-45127 Essen, Germany, email: sebastian.holthausen@uni-due.de
}, \qquad
Nina J. Husemann\thanks{Nina J. Husemann, University of Duisburg-Essen, Chair for Nonlinear Analysis and Modelling,  Faculty of Mathematics, Thea-Leymann-Stra{\ss}e 9,
	D-45127 Essen, Germany, email: nina.husemann@stud.uni-due.de}
\\[0.8em]
and \qquad
Patrizio Neff\thanks{
Patrizio Neff, University of Duisburg-Essen, Head of Chair for Nonlinear Analysis and Modelling, Faculty of Mathematics, Thea-Leymann-Stra{\ss}e 9,
D-45127 Essen, Germany, email: patrizio.neff@uni-due.de (corresponding author)}
}
\begin{document}
\maketitle
\vspace{-0,6cm}
\begin{abstract}
We recall in this note that the induced tangent stiffness tensor $\H^{\ZJ}_{\tau}(\tau)$ appearing in a hypoelastic formulation based on the Zaremba-Jaumann corotational derivative and the rate constitutive equation for the Kirchhoff stress tensor $\tau$ is minor and major symmetric if the Kirchhoff stress $\tau$ is derived from an elastic potential $\WW(F)$. This result is vaguely known in the literature. Here, we expose two different notational approaches which highlight the full symmetry of the tangent stiffness tensor $\H^{\ZJ}_{\tau}(\tau)$. The first approach is based on the direct use of the definition of each symmetry (minor and major), i.e., via contractions of the tensor with the deformation rate tensor $D$. The second approach aims at finding an absolute expression of the tensor $\H^{\ZJ}_{\tau}(\tau)$, by means of special tensor products and their symmetrisations. In some past works, the major symmetry of $\H^{\ZJ}_{\tau}(\tau)$ has been missed because not all necessary symmetrisations were applied.
\\
The analogous tangent stiffness tensor $\H^{\ZJ}(\sigma)$, relating the Cauchy stress tensor $\sigma$ to the Zaremba-Jaumann corotational derivative is also obtained, with both methods used for $\H^{\ZJ}_{\tau}(\tau)$.
\\
The approach is exemplified for the isotropic Hencky energy. Corresponding stability checks of software packages are shortly discussed.

\medskip

\\
Mathscinet classification
74B20 (Nonlinear elasticity)

\bigskip

\end{abstract}
\clearpage
\tableofcontents
\section{Introduction}
\subsection{Recap on rate-formulations for nonlinear elasticity}
Nonlinear elasticity is today an integral part of continuum mechanical modelling. The basic framework is definitely established, see e.g. \cite{doyle1956,Green1968a,Wang1973a,Washizu1975a,Eringen1980a,Ogden83,ciarlet2022,Marsden1983a,Antman1995a}. Here, we are interested in a question occurring when implementing a FEM-code for nonlinear elasticity.\\
While the overwhelming current trend is to use a totally Lagrangean approach, providing the direct discretization of the equilibrium equation
\begin{equation}
	\label{eq:intr_equilibirum_equation}
	\dvg \, S_1(F) = f
\end{equation}
on the reference configuration $\Omega \subset \R^3$ where $\phi : \Omega \subset \R^3 \to \R^3$ is the deformation, $F=\DD \phi \in \GL^+ (3)$ is the deformation gradient (the Fréchet-derivative of $\phi$) and $S_1(F)$ is the (non-symmetric) first Piola-Kirchhoff stress, many established software packages like Abaqus\textsuperscript{\texttrademark}, Ansys\textsuperscript{\texttrademark}\footnote{material stability check of Ansys\textsuperscript{\texttrademark}, release 18.2 (cf.\cite{Ansys_Release,Ansys_help}):\\
"Stability checks are provided for the Mooney-Rivlin hyperelastic materials. A nonlinear material is stable if the secondary work required for an arbitrary change in the deformation is always positive. Mathematically, this is equivalent to:
	\begin{equation}
		\label{eq:dsigmadepsilon}
		\sum \dif \sigma_{ij} \, \dif \varepsilon_{ij} > 0
	\end{equation}
	where \quad $\dif \sigma \, =$ change in the Cauchy stress tensor corresponding to a change $[\dif \varepsilon]$ in the logarithmic strain.\\
	Since the change in stress is related to the change in strain through the material stiffness tensor, checking for stability of a material can be more conveniently accomplished by
	checking for the positive definiteness of the material stiffness.\\
	The material stability checks are done at the end of preprocessing but before an analysis actually begins. At that time, the program checks for the loss of stability for six typical stress paths (uniaxial tension and compression, equibiaxial tension and compression, and planar tension and compression). The range of the stretch ratio over which the stability is checked is chosen from 0.1 to 10. If the material is stable over the range then no message will appear. Otherwise, a warning message appears that lists the Mooney-Rivlin constants and the critical values of the nominal strains where the material first becomes unstable."

	\begin{remark}
For non-homogeneous response, $\langle \dif \sigma ,  \dif \varepsilon \rangle$ is only the second-order work in linear elasticity (cf.~\cite[Appendix]{CSP2024}). Also we surmise that in \eqref{eq:dsigmadepsilon} it is really treated only the incompressible case, so that also Ansys\textsuperscript{\texttrademark} is checking, as Abaqus\textsuperscript{\texttrademark} does, Hill's inequality $\langle \dif \tau \, , \, \dif \varepsilon_{\log} \rangle > 0$ in the compressible case.\\
	\end{remark}
}
, LS-DYNA etc. discretise a rate-type formulation on the current configuration $\phi(\Omega)$ and the constitutive law is accordingly an expression relating increments of spatial stress to increments of spatial strain.
For example, Abaqus\textsuperscript{\texttrademark} in its UMAT subroutine requires the rate-formulation for the Kirchhoff stress $\tau$ expressed with the Zaremba-Jaumann rate.\\
Here, it has to be remarked that any hyper- or Cauchy-elastic law on the reference configuration can be equivalently expressed in a rate-type or hypoelastic format. More precisely, a hypo-elastic material, in the sense of Truesdell \cite{truesdellremarks} and Noll \cite{Noll55}, obeys a constitutive law of the form
	\begin{align}
	\label{eqthedoublestr}
	\frac{\DD^{\sharp}}{\DD t}[\sigma] = \H^{\sharp}(\sigma) . D \quad \iff \quad \underbrace{D}_{\text{the stretching}} = [\H^{\sharp}(\sigma)]^{-1} . \underbrace{\frac{\DD^{\sharp}}{\DD t}[\sigma]}_{\text{the stressing}} = \mathbb{S}^{\sharp}(\sigma) . \frac{\DD^{\sharp}}{\DD t}[\sigma],
	\end{align}
where
	\begin{itemize}
	\item $\frac{\DD^{\sharp}}{\DD t}[\sigma]$ is an appropriate \textbf{objective} rate of the Cauchy stress tensor $\sigma$,
	\item $\H^{\sharp}(\sigma)$ is a constitutive \textbf{fourth-order} tangent stiffness \textbf{tensor},
	\item $\mathbb{S}^{\sharp}(\sigma) = [\H^{\sharp}(\sigma)]^{-1}$ is a constitutive \textbf{fourth-order} tangent compliance \textbf{tensor} and
	\item $D = \sym \, \DD_{\xi} v  = \sym \, L$ is the \textbf{Eulerian strain rate tensor}, measuring the spatial rate of deformation \newline or ``stretching'', where $v$ describes the spatial velocity at a point $\xi$ in the current configuration, and $L = \DD_{\xi} v = \dot{F} \, F^{-1}$ is the velocity gradient.
	\end{itemize}
Hence, the elastic stretching $D$ depends only on the current stress level $\sigma$ and the rate of stress $\frac{\DD^{\sharp}}{\DD t}[\sigma]$ as seen by the objective derivative. \\
\\
As for time derivatives $\frac{\DD^{\sharp}}{\DD t}[\sigma]$ of the Cauchy stress tensor $\sigma$, there are many different possible choices (see e.g.~also the related works by
Xiao et al. \cite{xiao97, xiao98_2}, Bonet and Wood \cite{Bonet2008a},
Aubram \cite{Aubram2017}, Federico \cite{bellini2015, palizi2020consistent}, Fiala \cite{fiala2009,fiala2020objective}),
Govindjee \cite{govindjee1997}, Korobeynikov \cite{korobeynikov2018, korobeynikov2023}, Ortiz et al.~\cite{pinsky1983} and the book by Ogden \cite{Ogden83}). \\

This framework admits usually two basic choices. We can express the rate-form constitutive law in terms of the Cauchy stress $\sigma$, or alternatively in terms of the Kirchhoff stress (the weighted Cauchy stress)
\begin{equation}
	\tau = \det F \cdot \sigma = J \cdot \sigma \, .
\end{equation}
On the other hand, we are free to choose the objective derivatives \cite{pos_cor, leblond1992constitutive, CSP2024}. Provided the induced tangent stiffness tensor $\H^{\sharp}$ in \eqref{eqthedoublestr} is adjusted properly, all formulations are fully equivalent. Here, we consider primarily the combination of the Kirchhoff-stress $\tau$ with the \emph{corotational Zaremba-Jaumann rate} \cite{pos_cor,zaremba1903forme,jaumann1911geschlossenes}, using the vorticity tensor (or spin tensor) $W = \sk \, L, \; L = \dot F \, F^{-1}$,
	\begin{align}
	\label{ZJrate01}
	\boxed{\frac{\DD^{\ZJ}}{\DD t}[\sigma] := \frac{\DD}{\DD t}[\sigma] + \sigma \, W - W \, \sigma = Q \, \frac{\DD}{\DD t}[Q^T \, \sigma \, Q] \, Q^T  \quad \text{for $Q(t) \in \OO(3)$ \quad so that} \;  W =\dot{Q} \, Q^T\in \mathfrak{so}(3)}.
	\end{align}
In \cite{CSP2024} it is shown how to obtain the corresponding induced stiffness tensor $\H^{\ZJ}(\sigma)$ if $\sigma = \sigma (B)$ is given as an isotropic Cauchy elastic law. Indeed, for $\sigma (B) = \widehat{\sigma} (\log B)$ it holds
\begin{equation}
	\label{eq:HZJ}
	\H^{\ZJ} (\sigma). D \coloneqq \DD_B \sigma (B). \, [BD+DB] = \DD_{\log B} \, \widehat{\sigma} (\log B). \, \DD_B \log B. \, \frac{\DD^{\ZJ}}{\DD t}[B]
\end{equation}
together with
\begin{equation}
	\label{eq:const_law}
	\frac{\DD^{\ZJ}}{\DD t}[\sigma] = \H^{\ZJ} (\sigma). D \, .
\end{equation}
The constitutive law \eqref{eq:const_law} can be converted into a rate-formulation for the Kirchhoff stress $\tau$:
\begin{equation}
	\frac{\DD^{\ZJ}}{\DD t}[\tau] = \H^{\ZJ}_{\tau} (\tau). D \, .
\end{equation}
This will be recalled in the next section \ref{sec:converting_hypoelastic}. It is then an observation, that while the fourth-order tensor $\H^{\ZJ}(\sigma)$ is in general not major-symmetric (but always minor symmetric), the corresponding induced stiffness tensor $\H^{\ZJ}_{\tau}(\tau)$ for the Kirchhoff stress will be major-symmetric provided $\tau = \tau (B)$ derives from hyperelasticity.
\\
This will be shown by two different and complementing approaches and this constitutes the major result of this contribution. Demonstrating all symmetries of the stiffness tensor $\H^{\ZJ}(\sigma)$ is important for two reasons. First, from the epistemological point of view, if a symmetry exists, it \emph{must} be studied and accounted for. Second, from the computational point of view, a fourth-order tensor with major and both minor symmetries has 21 independent components, as opposed to the 36 of a fourth-order tensor with only the minor symmetries, and this reduces the memory allocation, which can be advantageous in large Finite Element models.
\\
In addition we will give an explicit example for the calculation of $\H^{\ZJ}(\sigma)$ and $\H^{\ZJ}_{\tau}(\tau)$ for the Hencky energy in section \ref{sec:H_ZJ_Hencky}. Finally, we shortly discuss a customary stability check of Abaqus\textsuperscript{\texttrademark} and Ansys\textsuperscript{\texttrademark} in relation to the Hencky model.

\subsection{Converting the hypoelastic formulation with Zaremba-Jaumann rate and Cauchy stress $\sigma$ into an equivalent hypoelastic formulation for the Kirchhoff stress $\tau$}
\label{sec:converting_hypoelastic}
Assume we know how to express the rate-formulation for the Cauchy stress $\sigma$ and the Zaremba-Jaumann rate, i.e. we have
\begin{equation}
		\frac{\DD^{\ZJ}}{\DD t}[\sigma] = \H^{\ZJ}(\sigma). D
\end{equation}
and we know the fourth-order tensor $\H^{\ZJ}(\sigma)$.
Then it is easy to obtain the rate-formulation for the Kirchhoff stress $\tau$, as is shown next.\\
We recall that $\tau = \det F \cdot \sigma = J \cdot \sigma = F \, S_2 \, F^T$. Then
	\begin{equation}
	\begin{alignedat}{2}
	\frac{\DD}{\DD t}[\tau(t)] &= \frac{\DD}{\DD t}[J \, \sigma] = \dot{J} \, \sigma + J \, \frac{\DD}{\DD t}[\sigma(t)] \\
	&=J \, \tr(D) \, \sigma + J \, \frac{\DD}{\DD t}[\sigma(t)] = J \, \left(\frac{\DD}{\DD t}[\sigma] + \tr(D) \, \sigma \right).
	\end{alignedat}
	\end{equation}
Moreover
	\begin{equation}
	\frac{\DD^{\ZJ}}{\DD t}[\sigma] = \frac{\DD}{\DD t}[\sigma] + \sigma \, W - W \, \sigma \qquad \text{and} \qquad \frac{\DD^{\ZJ}}{\DD t}[\sigma] = \H^{\ZJ}(\sigma) . D
	\end{equation}
and\footnote
{
Compare with the relation (eq. \eqref{eq2.113001})
	\begin{equation}
	\H_{\tau}(\tau) = \lambda \, \left[\H\left(\frac{1}{\lambda} \, \tau \right) + \frac{1}{\lambda} \, \tau \right]
	\end{equation}
between the uniaxial Kirchhoff stress induced stiffness tensor $\H_{\tau}(\tau)$ and the logarithmic stiffness $\H(\sigma)$ (cf. Appendix \ref{sec:special_case_1D_hypo_Kirchhoff_stress}).
}
	\begin{equation}
	\label{eq0.4}
	\begin{alignedat}{2}
	\frac{\DD^{\ZJ}}{\DD t}[\tau] &= \frac{\DD^{\ZJ}}{\DD t}[J \, \sigma] = \frac{\DD}{\DD t}[J \, \sigma] + (J \, \sigma) \, W - W \, (J \, \sigma) = J \, \frac{\DD}{\DD t}[\sigma] + \dot{J} \, \sigma + J \, \sigma \, W - J \, W \, \sigma \\
	&= J \left[ \frac{\DD}{\DD t}[\sigma] + \tr(D) \, \sigma \right] + J \, \sigma \, W - W \, J \, \sigma = J \left[ \frac{\DD}{\DD t}[\sigma] + \sigma \, W - W \, \sigma \right] + \tr(D) \, \tau \\
	&= J \, \H^{\ZJ}(\sigma) . D + \tr(D) \, \tau = J \, \underbrace{\left[\H^{\ZJ}\left(\frac{1}{J} \, \tau \right) . D + \frac{1}{J} \, \tr(D) \, \tau \right]}_{=: \; \mathbb{D}.D} =: \H^{\ZJ}_{\tau}(\tau;J).D \\
	&= J \, \left( \frac{\DD}{\DD t}[\sigma] + \sigma \, W - W \, \sigma + \tr(D) \, \sigma \right) = J \, \frac{\DD^{\text{Hencky}}}{\DD t}[\sigma] \\
	&= J \, \frac{\DD^{\text{Hencky}}}{\DD t}\left[\frac{1}{J} \tau \right], \quad \text{(cf. Korobeynikov \cite[p. 3879]{korobeynikov2023})}.
	\end{alignedat}
	\end{equation}

\noindent It is engineering folklore that $\H^{\ZJ}_{\tau}(\tau)$ is a symmetric matrix viewed as a (self-adjoint) mapping from \break $\R^6 \cong  \Sym(3) \to \Sym(3)$, i.e.~it has major symmetry if $\tau$ derives from an elastic potential (see e.g., Ji et al.~\cite[][Eq.~(33)]{ji2013}, Brannon~\cite{Brannon1998a}). In this respect we formulate the
	\begin{lemma}[Major symmetry for $\H^{\ZJ}_\tau$ in the hyperelastic case]
	For hyperelasticity, the induced Kirchhoff tangent stiffness tensor $\H^{\ZJ}_{\tau}(\tau)$ is major symmetric and $\mathbb{D} \coloneqq \frac{1}{J}\H^{\ZJ}_{\tau}(\tau)\in \Sym_4(6)$.
	\end{lemma}
\begin{proof}
Starting with the Zaremba-Jaumann derivative of the Kirchhoff stress $\tau$ we obtain
	\begin{equation}
	\label{eq0.5}
	\begin{alignedat}{2}
	\frac{\DD^{\ZJ}}{\DD t}[\tau] &= \frac{\DD}{\DD t}[\tau] + [\tau \ W - W \, \tau] = \frac{\DD}{\DD t}[F \, S_2 \, F^T] + [\tau \, W - W \, \tau] \\
	&= \dot{F} \, S_2 \, F^T + F \, \dot{S}_2 \, F^T + F \, S_2 \, \dot{F}^T + [\tau \, W - W \, \tau] \\
	&= F \, \dot{S}_2 \, F^T + \underbrace{\dot{F} \, F^{-1}}_{L} \, \underbrace{F \, S_2 \, F^T}_{\tau} + \underbrace{F \, S_2 \, F^T}_{\tau} \, \underbrace{F^{-T} \, \dot{F}^T}_{L^T} + [\tau \, W - W \, \tau] \\
	&= F \, \dot{S}_2 \, F^T + L \, \tau + \tau \, L^T + \tau \, W - W \, \tau = F \, \dot{S}_2 \, F^T + D \, \tau + \tau \, D, \qquad L = D + W \\
	&= F \, \DD_E S_2(E).[\dot E] \, F^T + D \, \tau + \tau \, D  \\
	&= F \, \DD_E S_2(E).[F^T \, D \, F] \, F^T + D \, \tau + \tau \, D =: \H^{\ZJ}_{\tau}(\tau).D.
	\end{alignedat}
	\end{equation}
In the second-last equation of \eqref{eq0.5} we used that by writing $S_2 = S_2(E)$ with $E = \frac12 \, (F^T \, F - \id)$, we obtain $\dot S_2(E) = \DD_E S_2(E). \dot E$, where
	\begin{align}
	\dot E = \frac12 \, (\dot F^T \, F + F^T \, \dot F) = \frac12 \, (F^T \, F^{-T} \, \dot F^T \, F + F^T \, \dot F \, F^{-1} \, F) = \frac12 \, (F^T \, L^T \, F + F^T \, L \, F) = F^T \, D \, F \, .
	\end{align}
Since major symmetry of $\H^{\ZJ}_{\tau}(\tau)$ amounts to $\langle \H^{\ZJ}_{\tau}(\tau).D_1, D_2 \rangle = \langle \H^{\ZJ}_{\tau}(\tau).D_2,D_1 \rangle$ and the mapping \break $\Sym(3) \to \Sym(3), \; D \mapsto \tau \, D + D \, \tau$ is already major symmetric, since with $\tau^T = \tau$
	\begin{align}
	\langle \tau \, D_1 + D_1 \, \tau, D_2 \rangle = \langle D_1 , \tau^T \, D_2 + D_2 \, \tau^T \rangle = \langle D_1 , \tau \, D_2 + D_2 \, \tau \rangle \, .
	\end{align}
It therefore remains to prove that the mapping $\Sym(3) \to \Sym(3), \; D \mapsto \; F \, \dot{S}_2 \, F^T = F \, \DD_E S_2(E).[F^T \, D \, F]$ is major symmetric, too. 
By assumption, the formulation is hyperelastic, which means that
	\begin{align}
	S_2(E) = \DD_E \WW(E),
	\end{align}
where $\WW$ is the elastic energy. Thus it follows that $\DD_E S_2(E) = \DD^2_E \WW(E)$ is itself self-adjoint, i.e. \break $(\DD_E S_2(E))^T = \DD_E S_2(E)$, meaning that
	\begin{align}
	\langle \DD_E S_2(E). H_1, H_2 \rangle = \langle \DD_E S_2(E). H_2, H_1 \rangle \qquad \forall \, H_1, H_2 \in \Sym(3).
	\end{align}
Therefore we obtain for the mapping $\Sym(3) \to \Sym(3), \; D \mapsto F \, \DD_E S_2(E).[F^T \, D \, F] \, F^T$
	\begin{equation}
	\begin{alignedat}{3}
	\langle F \, \DD_E S_2(E).[F^T \, D_1 \, F] \, F^T , D_2 \rangle &= \langle \DD_E S_2(E).[F^T \, D_1 \, F], F^T \, D_2 \, F \rangle 
	&&=\langle F^T \, D_1 \, F, (\DD_E S_2(E))^T.[F^T \, D_2 \, F] \rangle \\
	&=\langle F^T \, D_1 \, F, \DD_E S_2(E).[F^T \, D_2 \, F] \rangle
	&&= \langle D_1, F \, \DD_E S_2(E).[F^T \, D_2 \, F] \, F^T \rangle,
	\end{alignedat}
	\end{equation}
which is the required major symmetry.
\end{proof}
	\begin{remark}
	Minor symmetry of $\H^{\ZJ}_{\tau}(\tau)$ is clear from \eqref{eq0.5} since $\H^{\ZJ}_{\tau}(\tau).D \in \Sym(3), \; \forall \, D \in \Sym(3)$. Moreover, minor and major symmetry of $\H^{\ZJ}_{\tau}(\tau)$ are independent of assuming isotropy.

Looking at $\eqref{eq0.4}_3$ it is clear that major symmetry of $\H^{\ZJ}_{\tau}(\tau)$ does not imply that $\H^{\ZJ}(\sigma)$ is major symmetric, apart from the incompressible case in which $\tr(D) = 0$.
	\end{remark}
We observe that in the \textbf{incompressible case} we have $J=1$ and $\tr(D)$ = 0 and therefore
	\begin{align}
	\H_{\tau}^{\ZJ}(\tau;1) = \H^{\ZJ}(\tau) \qquad \text{together with} \qquad \frac{\DD^{\ZJ}}{\DD t}[\tau] = \H^{\ZJ}(\tau) . D
	\end{align}
so that the latter rate formulation for $\tau$ coincides with the Cauchy stress rate-formulation, as it should be for the incompressible case in which $\sigma = \tau$. \\
\\
There is also a more direct way to obtain the rate-formulation in terms of the Zaremba-Jaumann rate and the Kirchhoff stress $\tau$ in the case of isotropy. Indeed, for isotropy the Kirchhoff stress satisfies
	\begin{align}
	Q \, \tau(B) \, Q^T = \tau(Q \, B \, Q^T) \qquad \forall \, Q \in \OO(3).
	\end{align}
Therefore, we obtain (cf. \cite{CSP2024,pos_cor})
	\begin{equation}
	\label{eq1.18}
	\begin{alignedat}{2}
	\frac{\DD^{\ZJ}}{\DD t}[\tau] = \DD_B \tau(B) . [B \, D + D \, B] &=: \H^{\ZJ}_{\tau}(\tau). D \qquad \text{if} \quad B \mapsto \tau(B) \quad \text{is invertible} \\
	&= J \, \mathbb{D}(\tau). D \qquad \text{(Abaqus\textsuperscript{\texttrademark}-format)} .
	\end{alignedat}
	\end{equation}
According to the Abaqus\textsuperscript{\texttrademark} user manual, for 3D continuum elements Abaqus\textsuperscript{\texttrademark} uses the elasticity tensor associated with the Zaremba-Jaumann rate for the Kirchhoff stress $\tau$, i.e. the positive definiteness of $\H^{\ZJ}_{\tau}$ (i.e. $\mathbb{D} (\tau)$) is checked for material stability. \\

\subsection{Alternative approach and absolute expression of $\H^{\ZJ}_{\tau}(\tau)$}
\label{sec:Alternative_approach}
An absolute expression of the Zaremba-Jaumann elasticity tensor $\H^{\ZJ}_{\tau}(\tau)$, i.e., an expression not involving the double contraction with the deformation rate $D$, can be obtained using the special tensor products $\otimesdown$ and $\otimesup$ and, specifically, their symmetrisation $\otimesdownup$, as defined in the work by Curnier et al.~\cite{Curnier1995a}. We shall define these tensors for the case of Cartesian coordinates, but general covariant expressions can be found in \cite{Federico2012a,Federico2012b,bellini2015}.
As in the remainder of this work, $\mathbb{T} \ldot Z$ means that the fourth-order tensor $\mathbb{T}$ acts as a linear map on the second-order tensor $Z$. The transpose of the fourth-order tensor $\mathbb{T}$ is defined as the tensor $\mathbb{T}^T$ such that, for every second-order tensors $Y$ and $Z$,
\begin{align}
\label{eq:transpose_fourth-order}
\langle Y \,,\, \mathbb{T} \ldot Z \rangle = \langle Z \,,\, \mathbb{T}^T \ldot Y \rangle,
\qquad
Y_{ij} \, \mathbb{T}_{ijkl} \, Z_{kl} = Z_{kl} \, (\mathbb{T}^T)_{klij} \, Y_{ij}
\end{align}
In the literature on the special tensor products $\otimesdown$ and $\otimesup$ (e.g., \cite{Curnier1995a,Gasser2006b,Federico2012a,Federico2012b,bellini2015,Grillo2018a,palizi2020consistent}), the action $\mathbb{T} \ldot Z$ of $\mathbb{T}$ on $Z$ is usually denoted with the double contraction symbol, i.e., the colon ``$\dc$'', as $\mathbb{T} \dc Z$. In the same literature, the colon ``$\dc$'' is also used for the double contraction of two second-order tensors, $Y \dc Z$, which here is denoted by the scalar product $\langle Y , Z \rangle$.

\paragraph{Standard and special tensor products}
The standard tensor product $\otimes$ maps two second-order tensors $P$ and $Q$ into the fourth-order tensor $P \otimes Q$, defined by
\begin{align}
\label{eq:regular_tensor_product_action}
(P \otimes Q) \ldot Z            & = \langle Q , Z \rangle \, P                        &
(P \otimes Q)_{ijkl} \, Z_{kl} & = (Q_{kl} \, Z_{kl}) \, P_{ij},
\end{align}
for every second-order tensor $Z$. Therefore, the component expression of $P \otimes Q$ is the familiar
\begin{align}
\label{eq:regular_tensor_product_comp}
(P \otimes Q)_{ijkl} = P_{ij} \, Q_{kl}P_{ij}.
\end{align}
It is easy to verify, in components or with the definition~\eqref{eq:regular_tensor_product_action} of tensor product and the definition~\eqref{eq:transpose_fourth-order} of transpose of a fourth-order tensor, that 
\begin{align}
\label{eq:regular_tensor_product_transpose}
(P \otimes Q)^T = Q \otimes P.
\end{align}
Indeed, using, e.g., \eqref{eq:regular_tensor_product_action} and~\eqref{eq:transpose_fourth-order}, we have
\begin{align}
    \langle Y \,,\, (P \otimes Q) \ldot Z \rangle
& = \langle Y \,,\, \langle Q , Z \rangle \, P \rangle
  = \langle Y , P \rangle \, \langle Q , Z \rangle
  = \langle Z , Q \rangle \, \langle P , Y \rangle   \nonumber\\
& = \langle Z \,,\, \langle P , Y \rangle \, Q \rangle
  = \langle Z \,,\, (Q \otimes P) \, Y \rangle
  = \langle Z \,,\, (P \otimes Q)^T \, Y \rangle \, ,
\end{align}
which, for the arbitrariness of $Y$ and $Z$, implies~\eqref{eq:regular_tensor_product_transpose}.

The special tensor product $\otimesdown$ (denoted $\boxtimes$ by some authors, e.g., \cite{DelPiero1979a,Lucchesi1992a,Jog2006a,korobeynikov2018,korobeynikov2023}) maps two second-order tensors $P$ and $Q$ into the fourth-order tensor $P \otimesdown Q$ defined by
\begin{align}
\label{eq:tensordown_action}
(P \otimesdown Q) \ldot Z            & = P \, Z \, Q^T,                &
(P \otimesdown Q)_{ijkl} \, Z_{kl} & = P_{ik} \, Z_{kl} \, Q_{jl},
\end{align}
for every second-order tensor $Z$. The component expression of $P \otimesdown Q$ is thus
\begin{align}
\label{eq:tensordown_comp}
(P \otimesdown Q)_{ijkl} & = P_{ik} \, Q_{jl}.
\end{align}
For the case of the special tensor product $\otimesdown$, the transpose is
\begin{align}
\label{eq:tensordown_transpose}
(P \otimesdown Q)^T = P^T \otimesdown Q^T,
\end{align}
which can be verified in components or by the definition~\eqref{eq:tensordown_action} of $\otimesdown$ and the general definition~\eqref{eq:transpose_fourth-order} of transpose of a fourth-order tensor.

The special tensor product $\otimesup$ maps two second-order tensors $P$ and $Q$ into the fourth-order tensor $P \otimesup Q$ defined by
\begin{align}
\label{eq:tensorup_action}
(P \otimesup Q) \ldot Z            & = P \, Z^T Q^T,                &
(P \otimesup Q)_{ijkl} \, Z_{kl} & = P_{il} \, Z_{kl} \, Q_{jk},
\end{align}
for every second-order tensor $Z$. The component expression of $P \otimesup Q$ is thus
\begin{align}
\label{eq:tensorup_comp}
(P \otimesup Q)_{ijkl} & = P_{il} \, Q_{jk}.
\end{align}
For the case of the special tensor product $\otimesup$, the transpose is
\begin{align}
\label{eq:tensorup_transpose}
(P \otimesup Q)^T = Q^T \otimesup P^T,
\end{align}
which, again, can be verified in components or by the definition~\eqref{eq:tensorup_action} of $\otimesup$ and the general definition~\eqref{eq:transpose_fourth-order}.

\paragraph{Symmetrisation of the special tensor products}
In general, the special tensor products $\otimesdown$ and $\otimesup$ possess neither the major symmetry nor the minor symmetries. However, symmetrisations exist that turn out to be very useful.

The tensor product $\otimesdownup$ (denoted $\odot$ by some authors, e.g., \cite{Holzapfel2000a}) is defined by
\begin{equation}
\label{eq:tensordownup}
P \otimesdownup Q          = \tfrac{1}{2} (P \otimesdown Q + Q \otimesup P), \qquad
(P \otimesdownup Q)_{ijkl} = \tfrac{1}{2} (P_{ik} \, Q_{jl} + Q_{il} \, P_{jk}).
\end{equation}
The transpose of the fourth-order $P \otimesdownup Q$ is \emph{not} obtainable via $\otimesdownup$. Indeed, using~\eqref{eq:tensordown_transpose} and~\eqref{eq:tensorup_transpose}, we have
\begin{equation}
\label{eq:tensordownup_transpose}
(P \otimesdownup Q)^T = \tfrac{1}{2} (P^T \otimesdown Q^T + P^T \otimesup Q^T ).
\end{equation}
The tensor $P \otimesdownup Q$ possesses the \emph{left} minor symmetry \cite{bellini2015}, i.e., the minor symmetry on the first pair of indices. This can be verified easily in components, or by studying the action of $P \otimesdownup Q$ on an arbitrary second-order tensor $Z$:
\begin{equation}
\label{eq:tensordownup_left_minor_sym}
(P \otimesdownup Q) \ldot Z
 = \tfrac{1}{2} (P \, Z \, Q^T + Q \, Z^T P^T) = \tfrac{1}{2} (P \, Z \, Q^T + (P \, Z Q^T)^T)  = \mathrm{sym}(P \, Z \, Q^T) = \mathrm{sym}[(P \otimesdown Q) \ldot Z].
\end{equation}
In contrast, $P \otimesdownup Q$ does \emph{not} possess the right minor symmetry and, again, this can be verified in components, or by studying the action of its transpose $(P \otimesdownup Q)^T$ on an arbitrary second-order tensor $Y$:
\begin{align}
\label{eq:tensordownup_no_right_minor_sym}
(P \otimesdownup Q)^T \ldot Y
& = \tfrac{1}{2} (P^T \otimesdown Q^T + P^T \otimesup Q^T ) . Y
  = \tfrac{1}{2} (P^T Y \, Q + P^T Y^T Q)                       \nonumber\\
& = P^T \otimesup Q^T \ldot [\tfrac{1}{2} (Y^T + Y)]
  = P^T \otimesup Q^T \ldot \mathrm{sym}(Y).
\end{align}
The further symmetrisation
\begin{align}
\label{eq:tensordownup_symmetrised}
P \otimesdownup Q + Q \otimesdownup P
& = \tfrac{1}{2} (P \otimesdown Q + Q \otimesup P)
  + \tfrac{1}{2} (Q \otimesdown P + P \otimesup Q),  \\
(P \otimesdownup Q + Q \otimesdownup P)_{ijkl}
& = \tfrac{1}{2} (P_{ik} \, Q_{jl} + Q_{il} \, P_{jk})
  + \tfrac{1}{2} (Q_{ik} \, P_{jl} + P_{il} \, Q_{jk})  \nonumber
\end{align}
possesses \emph{both} minor symmetries. Indeed, the left minor symmetry follows from the fact that each of $P \otimesdownup Q$ and $Q \otimesdownup P$ possesses the minor symmetry, by virtue of~\eqref{eq:tensordownup_left_minor_sym}, and the right minor symmetry can be verified easily in components or, as above, by studying the action of the transpose $(P \otimesdownup Q + Q \otimesdownup P)^T$, which is easily obtainable from~\eqref{eq:tensordownup_transpose}, on an arbitrary second-order tensor $Y$:
\begin{align}
\label{eq:tensordownup_symmetrised_right_minor_sym}
(P \otimesdownup Q + Q \otimesdownup P)^T \ldot Y
& = \tfrac{1}{2} (P^T \otimesdown Q^T + P^T \otimesup Q^T + Q^T \otimesdown P^T + Q^T \otimesup P^T) \ldot Y
\nonumber\\
& = \tfrac{1}{2} (P^T Y \, Q + P^T Y^T Q + Q^T Y \, P + Q^T Y^T P)
\nonumber\\
& = \tfrac{1}{2} (P^T Y \, Q + Q^T Y^T P) + \tfrac{1}{2} (P^T Y^T Q + Q^T Y \, P)
\nonumber\\
& = \tfrac{1}{2} ((Q^T Y^T P)^T + Q^T Y^T P) + \tfrac{1}{2} (P^T Y^T Q + (P^T Y^T Q)^T)
\nonumber\\
& = \mathrm{sym}(Q^T Y^T P + P^T Y^T Q) = \mathrm{sym}[(Q^T \otimesup P^T + P^T \otimesup Q^T) \ldot Y]\, .
\end{align}
Finally, if both $P$ and $Q$ are symmetric, then $P \otimesdownup Q + Q \otimesdownup P$ also possesses the major symmetry, i.e., for every second-order tensors $Y$ and $Z$,
\begin{align}
\label{eq:tensordownup_symmetrised_major_sym}
\langle Y \,,\, (P \otimesdownup Q + Q \otimesdownup P) \ldot Z \rangle = \langle Z \,,\, (P \otimesdownup Q + Q \otimesdownup P) \ldot Y \rangle \, .
\end{align}
As usual, this can be verified in components, or by working out the (double) contractions.

\paragraph{Lie derivative of the Kirchhoff stress and spatial elasticity tensor} 
The Lie derivative of the Kirchhoff stress $\tau$ is related directly to the material elasticity tensor
\begin{align}
\label{eq:material_elasticity tensor}
\mathbb{C} := \DD_E S_2(E)
\end{align}
as already seen in Eq.~\eqref{eq0.5}. Indeed, the material elasticity tensor~\eqref{eq:material_elasticity tensor} features in the rate expression of the second Piola-Kirchhoff stress,
\begin{align}
\label{eq:rate_second_Piola}
\dot{S}_2 = \DD_E S_2(E) \ldot \dot{E} = \mathbb{C} \ldot \dot{E}
\end{align}
and, since the second Piola-Kirchhoff stress $S_2$ is the pull-back of the Kirchhoff stress $\tau$, i.e.,
\begin{align}
\label{eq:tau_and_second_Piola}
S_2 = F^{-1} \tau \, F^{-T}, 
\end{align}
the push-forward of the time derivative~\eqref{eq:rate_second_Piola} of $S_2$ is, by definition, the Lie derivative of the Kirchhoff stress $\tau$, as per Eq.~\eqref{eq:Lie_derivative_tau}. Therefore,
\begin{align}
\label{eq:Lie_derivative_tau_material_elasticity_tensor}
\mathcal{L}_v \tau =
F \left[ \frac{\DD}{\DD t} [F^{-1} \tau \, F^{-T}] \right] F^T =
F \, \dot{S}_2 F^T =
F \, [\mathbb{C} \ldot \dot{E}] \, F^T.
\end{align}
We aim at expressing the Lie derivative~\eqref{eq:Lie_derivative_tau_material_elasticity_tensor} as a function of the \emph{spatial elasticity tensor},
\begin{equation}
\label{eq:spatial_elasticity_tensor}
\mathbbs{C}        = J^{-1} \, (F \otimesdown F) \ldot \left[ \mathbb{C} \ldot (F^T \otimesdown F^T) \right] ,  \qquad
\mathbbs{C}_{ijkl}  = J^{-1} \, F_{iI} \, F_{jJ} \, \mathbb{C}_{IJKL} \, F_{kK} \, F_{lL},
\end{equation}
which is the full backward Piola transform of the material elasticity tensor $\mathbb{C}$, and relates the Truesdell rate (cf. Appendix \ref{sec:Some_objective_derivatives}) of the Cauchy stress $\sigma$ to the deformation rate $D$ \cite{Marsden1983a,Bonet2008a,Federico2012a,Federico2012b}. Note how the use of the special tensor product $\otimesdown$ allows to write the backward Piola transformation~\eqref{eq:spatial_elasticity_tensor} in component-free formalism, via the tensor $F \otimesdown F$ and its transpose $(F \otimesdown F)^T = F^T \otimesdown F^T$ \cite{Ateshian2010a,Federico2012b}. The inverse relation of~\eqref{eq:spatial_elasticity_tensor} is
\begin{equation}
\label{eq:material_from_spatial_elasticity_tensor}
\mathbb{C}         = J \, (F^{-1} \otimesdown F^{-1}) \ldot \left[ \mathbbs{C} \ldot (F^{-T} \otimesdown F^{-T}) \right],  \qquad
\mathbb{C}_{IJKL}  = J \, F^{-1}_{Ii} \, F^{-1}_{Jj} \, \mathbbs{C}_{ijkl} \, F^{-1}_{Kk} \, F^{-1}_{Ll}
\end{equation}
It is easy to show (for instance, in components), that substitution of~\eqref{eq:material_from_spatial_elasticity_tensor} into~\eqref{eq:Lie_derivative_tau_material_elasticity_tensor} yields
\begin{align}
\label{eq:Lie_derivative_tau_spatial_elasticity_tensor}
\mathcal{L}_v \tau = J \, \mathbbs{C} \ldot D \, ,
\end{align}
where we exploited the fact that the rate of deformation $D$ is the push-forward of the time derivative $\dot{E}$ of the Green-Lagrange strain $E$:
\begin{align}
\label{eq:deformation_rate_and_E_dot}
D = F^{-T} \dot{E} \, F^{-1}.
\end{align}

\paragraph{Zaremba-Jaumann rate of the Kirchhoff stress and its conjugated elasticity tensor}
Here, we use the Cartesian expressions of some results from Appendix~\ref{appendix:Jaumann_and_Lie}, which were obtained in covariant formalism. Substitution of the expression~\eqref{eq:Lie_derivative_tau_spatial_elasticity_tensor} of the Lie derivative of $\tau$ into the Zaremba-Jaumann rate~\eqref{eq:Jaumann_rate_and_Lie_derivative_tau} of $\tau$ yields
\begin{align}
\label{eq:Jaumann_rate_tau_spatial_elasticity_tensor}
\frac{\DD^{\ZJ}}{\DD t}[\tau] = J \, \mathbbs{C} \ldot D + D \, \tau + \tau \, D.
\end{align}
In order to find an absolute expression of the elasticity tensor $\H^{\ZJ}_{\tau}(\tau)$ relating the Zaremba-Jaumann rate of $\tau$ to the deformation rate $D$, we need to express the term $D \, \tau + \tau \, D$ as a fourth-order tensor double-contracted with $D$. We note that, by definition of identity tensor $\id$ (which, in this case, means the \emph{spatial} identity tensor and, rigorously speaking, would be the inverse metric tensor $g^{-1}$, which has contravariant components $g^{ij}$, similarly to the Kirchhoff stress $\tau$, which has contravariant components $\tau^{ij}$), we have the identity
\begin{align}
D \, \tau + \tau \, D = \id \, D \, \tau + \tau \, D \, \id.
\end{align}
By virtue of the symmetry all the tensors involved, i.e., $\id^T = \id$, $D^T = D$, $\tau^T = \tau$, this can be written 
\begin{align}
\label{eq:identity_relation}
D \, \tau + \tau \, D & = \tfrac{1}{2} \, (\id \, D \, \tau^T + \id \, D^T \tau^T)
                        + \tfrac{1}{2} \, (\tau \, D \, \id^T + \tau \, D^T \id^T) \nonumber\\
                      & = \tfrac{1}{2} \, (\id \, D \, \tau^T + \tau \, D^T \id^T)
                        + \tfrac{1}{2} \, (\tau \, D \, \id^T + \id \, D^T \tau^T).
\end{align}
In Eq.~\eqref{eq:identity_relation}, we recognise the symmetrised special tensor product $\id \otimesdownup \tau + \tau \otimesdownup \id$ defined in Eq.~\eqref{eq:tensordownup_symmetrised}, i.e., we can write~\eqref{eq:identity_relation} as
\begin{align}
\label{eq:identity_relation_otimesdownup_tau}
D \, \tau + \tau \, D = [ \id \otimesdownup \tau + \tau \otimesdownup \id ] \ldot D.
\end{align}
Substitution of~\eqref{eq:identity_relation_otimesdownup_tau} into~\eqref{eq:Jaumann_rate_tau_spatial_elasticity_tensor} yields
\begin{align}
\frac{\DD^{\ZJ}}{\DD t}[\tau] = [ J \, \mathbbs{C} + \id \otimesdownup \tau + \tau \otimesdownup \id ] \ldot D,
\end{align}
which compared with~\eqref{eq0.5}, finally gives the sought absolute expression of $\H^{\ZJ}_{\tau}(\tau)$, as
\begin{align}
\label{eq:absolute_Jaumann-rate_tau_elasticity_tensor}
\H^{\ZJ}_{\tau}(\tau) = J \, \mathbbs{C} + \id \otimesdownup \tau + \tau \otimesdownup \id \, ,
\end{align}
or, in components:
\begin{align}
\label{eq:absolute_Jaumann-rate_tau_elasticity_tensor_comp}
[\H^{\ZJ}_{\tau}(\tau)]_{ijkl}
= J \, \mathbbs{C}_{ijkl}
+ \tfrac{1}{2} (\delta_{ik} \, \tau_{jl} + \tau_{il} \, \delta_{jk})
+ \tfrac{1}{2} (\tau_{ik} \, \delta_{jl} + \delta_{il} \, \tau_{jk}) \, .
\end{align}
The elasticity tensor $\H^{\ZJ}_{\tau}(\tau)$ defined in Eq.~\eqref{eq:absolute_Jaumann-rate_tau_elasticity_tensor} possesses the major symmetry and both minor symmetries, which follow from those of $\mathbbs{C}$ and $ \id \otimesdownup \tau + \tau \otimesdownup \id$ (the major symmetry of $\id \otimesdownup \tau + \tau \otimesdownup \id$ descends from the symmetry of $\tau$ and $\id$).

The spatial elasticity tensor $\mathbbs{C}$ inherits the major symmetry and both minor symmetries from the material elasticity tensor $\mathbb{C}$, which, in turn, enjoys the major symmetry because of the smoothness of the elastic potential (Schwarz' theorem on the symmetry of the second partial derivatives), and both minor symmetries because of the symmetry of the Green-Lagrange strain $E$. The symmetrised special tensor product $\id \otimesdownup \tau + \tau \otimesdownup \id$ possesses both minor symmetries, as shown in Eqs.~\eqref{eq:tensordownup_left_minor_sym} and~\eqref{eq:tensordownup_symmetrised_right_minor_sym}, and, by virtue of the symmetry of $\tau$ and $\id$, it also possesses the major symmetry, as shown in Eq.~\eqref{eq:tensordownup_symmetrised_major_sym}. In a past work~\cite{palizi2020consistent}, the major symmetry was missed because, in place of the symmetrised contribution $\id \otimesdownup \tau + \tau \otimesdownup \id$, the contribution $\tau \otimesdown \id + \id \otimesup \tau = 2 (\tau \otimesdownup \id)$ was used. This was the result of not accounting explicitly for the symmetry of $\tau$ and $\id$.

We remark that the component expression~\eqref{eq:absolute_Jaumann-rate_tau_elasticity_tensor_comp} is particularly convenient for the numerical implementation of $\H^{\ZJ}_{\tau}(\tau)$, regardless of the specific software package at hand. As mentioned in Section~\ref{sec:converting_hypoelastic}, the programming of Abaqus\textsuperscript{\texttrademark} \texttt{UMAT} user-defined material subroutines requires the elasticity tensor
\begin{align}
\label{eq:absolute_Jaumann-rate_tau_elasticity_tensor_Abaqus}
\mathbb{D}(\tau) = J^{-1} \, \H^{\ZJ}_{\tau}(\tau) = \mathbbs{C} + J^{-1} \, (\id \otimesdownup \tau + \tau \otimesdownup \id).
\end{align}
Lastly, we note again that, in covariant formalism, the identity $\id$ must be replaced by the inverse metric tensor $g^{-1}$.

\paragraph{Truesdell rate of the Cauchy stress and spatial elasticity tensor} 
Similarly to the case of the Lie derivative of the Kirchhoff stress $\tau$, the Truesdell rate of the Cauchy stress is related directly to the material elasticity tensor $\mathbb{C}$ of Eq.~\eqref{eq:material_elasticity tensor}. Indeed, the second Piola-Kirchhoff stress $S_2$ is the the backward Piola transform of the Cauchy stress $\sigma$, i.e.,
\begin{align}
\label{eq:sigma_and_second_Piola}
S_2 = J \, F^{-1} \sigma \, F^{-T}, 
\end{align}
and the forward Piola transform of the time derivative~\eqref{eq:rate_second_Piola} of $S_2$ is, by definition, the Truesdell rate of the Cauchy stress $\sigma$, as per Eq.~\eqref{eq:Truesdell_rate_sigma_Piola_transforms}. Therefore,
\begin{align}
\label{eq:Truesdell_rate_sigma_material_elasticity_tensor}
\frac{\DD^{\TR}}{\DD t} [\sigma] =
J^{-1} F \left[ \frac{\DD}{\DD t} [J \, F^{-1} \sigma \, F^{-T}] \right] F^T =
J^{-1} F \, \dot{S}_2 F^T =
J^{-1} F \, [\mathbb{C} \ldot \dot{E}] \, F^T.
\end{align}
With the same procedure seen above for the Lie derivative $\mathcal{L}_v \tau$ of the Kirchhoff stress $\tau$, we can show that the Truesdell rate~\eqref{eq:Truesdell_rate_sigma_material_elasticity_tensor} of the Cauchy stress $\sigma$ is related to the spatial elasticity tensor $\mathbbs{C}$ of~\eqref{eq:spatial_elasticity_tensor} via
\begin{align}
\label{eq:Truesdell_rate_sigma_spatial_elasticity_tensor}
\frac{\DD^{\TR}}{\DD t} [\sigma] = \mathbbs{C} \ldot D \, .
\end{align}

\paragraph{Zaremba-Jaumann rate of the Cauchy stress and its conjugated elasticity tensor}
We refer again to Appendix~\ref{appendix:Jaumann_and_Lie}. Substitution of the expression~\eqref{eq:Truesdell_rate_sigma_spatial_elasticity_tensor} of the Truesdell rate of $\sigma$ into the Zaremba-Jaumann rate~\eqref{eq:Jaumann_rate_and_Truesdell_rate_sigma} of $\sigma$ yields
\begin{align}
\label{eq:Jaumann_rate_sigma_spatial_elasticity_tensor}
\frac{\DD^{\ZJ}}{\DD t}[\sigma] = \mathbbs{C} \ldot D + D \, \sigma + \sigma \, D - \tr(D) \, \sigma \, .
\end{align}
The term $D \, \sigma + \sigma \, D$ is obtained by simply dividing Eq.~\eqref{eq:identity_relation_otimesdownup_tau} by $J$, as
\begin{align}
\label{eq:identity_relation_otimesdownup_sigma}
D \, \sigma + \sigma \, D = [ \id \otimesdownup \sigma + \sigma \otimesdownup \id ] \ldot D,
\end{align}
and the term $\tr(D) \, \sigma$ is obtained via the definitions of the trace and (standard) tensor product, as
\begin{align}
\label{eq:identity_relation_trace_D}
\tr(D) \, \sigma = \langle \id , D \rangle \, \sigma = (\sigma \otimes \id) \ldot D \, .
\end{align}
Substitution of~\eqref{eq:identity_relation_otimesdownup_sigma} and~\eqref{eq:identity_relation_trace_D} into~\eqref{eq:Jaumann_rate_sigma_spatial_elasticity_tensor}
yields
\begin{align}
\frac{\DD^{\ZJ}}{\DD t}[\sigma] = [ \mathbbs{C} + \id \otimesdownup \sigma + \sigma \otimesdownup \id - (\sigma \otimes \id) ] \ldot D \, ,
\end{align}
which, compared with \eqref{eq:HZJ} and \eqref{eq:HZJ2}
 yields the absolute expression of $\H^{\ZJ}(\sigma)$ as
\begin{align}
\label{eq:absolute_Jaumann-rate_sigma_elasticity_tensor}
\H^{\ZJ}(\sigma) = \mathbbs{C} + \id \otimesdownup \sigma + \sigma \otimesdownup \id - \sigma \otimes \id \, ,
\end{align}
or, in components,
\begin{align}
\label{eq:absolute_Jaumann-rate_sigma_elasticity_tensor_comp}
[\H^{\ZJ}(\sigma)]_{ijkl}
= \, \mathbbs{C}_{ijkl}
+ \tfrac{1}{2} (\delta_{ik} \, \sigma_{jl} + \sigma_{il} \, \delta_{jk})
+ \tfrac{1}{2} (\sigma_{ik} \, \delta_{jl} + \delta_{il} \, \sigma_{jk})
- \sigma_{ij} \, \delta_{kl} \, .
\end{align}
All considerations made for $\H^{\ZJ}_{\tau}(\tau)$ hold for $\H^{\ZJ}_{\sigma}(\sigma)$, except for the fact that $\H^{\ZJ}_{\sigma}(\sigma)$ does \emph{not} posses major symmetry, due to the presence of the term $\sigma \otimes \id$ which can only be major symmetric when the Cauchy stress $\sigma$ is hydrostatic: not a very interesting scenario, and certainly not a general scenario. Lastly, we note again that, in covariant formalism, the identity $\id$ must be replaced by the inverse metric tensor $g^{-1}$.\\

As for the relevance of the representation \eqref{eq:absolute_Jaumann-rate_sigma_elasticity_tensor_comp} we observe that indeed, in~\cite[Eq.~(6)]{Li1998a}, in~\cite[Eq.~(2.22)]{Wang1995a} and in~\cite[Eq.~(22)]{Morris2000a}, the following (largely ad hoc) constitutive stiffness tensor is analysed for stability of the elastic law:
\begin{equation}
	\H_{ijkl} (\sigma) = \mathbbs{C}_{ijkl} + \tfrac{1}{2} (\delta_{ik} \, \sigma_{jl} + \sigma_{il} \, \delta_{jk} + \sigma_{ik} \, \delta_{jl} + \delta_{il} \, \sigma_{jk} - 2 \, \delta_{kl} \, \sigma_{ij}) 
\end{equation}
where $\mathbbs{C}$ is interpreted by us as the spatial elasticity tensor occurring in~\eqref{eq:Truesdell_rate_sigma_spatial_elasticity_tensor} (\cite[eq.~(2.22)]{Wang1995a}) and the authors define [material] instability to be present at $\det \H(\sigma) = 0$.
Taking into account eq.~\eqref{eq:absolute_Jaumann-rate_sigma_elasticity_tensor_comp} it occurs therefore that the above $\H_{ijkl}$ coincides with the induced tangent stiffness matrix appearing in the rate-formulation for the Cauchy stress $\sigma$ and the Zaremba-Jaumann rate \cite{CSP2024}
\begin{equation}
	\label{eq:HZJ2}
	\frac{\DD^{\ZJ}}{\DD t} [\sigma] = \H^{\ZJ} (\sigma) . D \, , \quad \quad \H^{\ZJ} (\sigma) . D \colonequals \DD_B \sigma(B) . [B \, D + D \, B] \, .
\end{equation}
The condition $\det \H^{\ZJ} (\sigma) = 0$ is, in fact, equivalent to checking (cf.~\cite{CSP2024}) the local invertibility of $B \mapsto \sigma (B)$, i.e.
\begin{equation}
	\label{eq:detDsigma=0}
	\det \DD_B \sigma (B) = 0 \, .
\end{equation}
In the paper \cite{Wang1995a} the authors check, in this interpretation,
\begin{equation}
	\label{eq:detsym}
	\det (\sym \, \H^{\ZJ} (\sigma)) = 0
\end{equation}
being aware that $\H^{\ZJ}(\sigma)$ is, in general, not major symmetric and we note that \cite{leblond1992constitutive, agostino2024,secondorderwork2024}
\begin{equation}
	\label{eq:Leblond_sym_log}
	\langle \sym \, \H^{\ZJ} (\sigma) . D , D \rangle > 0 \quad \iff \quad \DD_{\log V} \sigma (\log V) \in \Sym_4^{++}(6) \quad \overset{\text{(Def)}}{\iff} \quad \text{TSTS-M$^{++}$}\, .
\end{equation}
Moreover by the Bendixson Theorem \cite{bendixson1902} \eqref{eq:Leblond_sym_log} implies $\det \H^{\ZJ}(\sigma)>0$, hence excluding \eqref{eq:detDsigma=0}.\\

\section{The induced tangent stiffness tensor $\H^{\ZJ}_{\tau}(\tau)$ for the Hencky energy}
\label{sec:H_ZJ_Hencky}
\subsection{The compressible Hencky energy}
Let us now present as an example of the foregoing development the Hencky model.
The Hencky energy (see \cite{Neff_Osterbrink_Martin_Hencky13, NeffGhibaLankeit, NeffGhibaPoly, xiao97_2, xiao97, xiao98_2, xiao98_1} for related literature) reads
	\begin{align}
	\WW_{\text{Hencky}}(V) = \widehat{\WW}_{\text{Hencky}} (\log V) = \mu \, \norm{\log V}^2 + \frac{\lambda}{2} \, \tr^2(\log V) = \frac{\mu}{2} \, \norm{\log B}^2 + \frac{\lambda}{4} \, (\log \det B)^2 \, .
	\end{align}
The corresponding Kirchhoff stress $\tau$ can be calculated using the Richter-representation \cite{richtertranslation, richter1948isotrope, richter1949hauptaufsatze, Richter50, Richter52,moreau1979lois,vallee1978lois,vallee2008dual}
	\begin{align}
	\tau = \DD_{\log V} \widehat \WW(\log V) = 2 \, \mu \, \log V + \lambda \, \tr(\log V) \, \id = \mu \, \log B + \frac{\lambda}{2} \, \tr(\log B) \, \id.
	\end{align}
For $\mu, \, 2\mu + 3 \lambda > 0$ Hill's inequality is satisfied (cf.~{\cite{hill1968constitutivea,hill1968constitutiveb,hill1970constitutive,sidoroff1974restrictions})
	\begin{align}
	\label{eq:Hills_inequality}
	\langle \tau(\log V_1) - \tau(\log V_2), \log V_1 - \log V_2 \rangle > 0 \qquad \forall \, V_1, V_2 \in \Sym^{++}(3), \quad V_1 \neq V_2
	\end{align}
and we even have
	\begin{align}
	\label{eq:derivative_tau}
	\DD_{\log V} \tau(\log V) = \DD^2_{\log V} \widehat \WW_{\text{Hencky}}(\log V) \in \Sym^{++}_4(6) \, ,
	\end{align}
since
\begin{equation}
	\langle \DD_{\log V} \, \tau (\log V). \, H , H \rangle = \langle 2 \, \mu H + \lambda \, \tr (H) \, \id , H \rangle
	= 2 \, \mu \, \Vert H \Vert^2 + \lambda \, \tr^2 (H) > 0
\end{equation}
and $\widehat{\WW}_{\text{Hencky}} (\log V)$ is convex in $\log V$.\\

We wish to directly determine the rate-formulation for the Kirchhoff stress $\tau$ with the Zaremba-Jaumann rate. From \cite{agostino2024} and \cite{CSP2024,pos_cor} we have 
	\begin{equation}
	\label{eqhenckytensor1}
	\begin{alignedat}{2}
	\frac{\DD^{\ZJ}}{\DD t}[\tau] &= \H^{\ZJ}_{\tau}(\tau).D = \frac{\DD^{\ZJ}}{\DD t}[\widehat \tau(\log B)] \\
	&= \DD_{\log B} \widehat \tau(\log B). \frac{\DD^{\ZJ}}{\DD t}[\log B] = \DD_{\log B} \widehat \tau(\log B).(\DD_B \log B.[B \, D + D \, B]) \\
	&= \mu \, \underbrace{\DD_B \log B.[D \, B + D \, B]}_{\text{self-adjoint}} + \frac{\lambda}{2} \, \underbrace{\tr(\DD_B \log B.[B \, D + D \, B])}_{= \, 2 \, D \, (\text{cf. \cite{CSP2024}})} \, \id \\
	&= \mu \, \left[\frac{\DD^{\ZJ}}{\DD t}[\log B]\right] + \frac{\lambda}{2} \, \tr \left(\frac{\DD^{\ZJ}}{\DD t}[\log B]\right) \, \id = \mu \, (\DD_B \log B. [B \, D + D \, B]) + \lambda \, \tr(D) \, \id \\
	&= 2 \, \mu \, \left[\DD_B \log B. \left[\frac{B \, D + D \, B}{2} \right] \right] + \lambda \, \tr(D) \, \id =: \H^{\ZJ}_{\tau}(\tau).D \, .
	\end{alignedat}
	\end{equation}
It is easy to see directly that $\H^{\ZJ}_{\tau}(\tau)$ is self-adjoint since the first part in $\eqref{eqhenckytensor1}_3$ is, and $\H^{\ZJ}_{\tau}(\tau)$ is obviously minor symmetric. Moreover, $\H^{\ZJ}_{\tau}(\tau)$ is positive-definite for $\mu, \, 2\mu + 3 \lambda > 0$ since (cf. \cite{CSP2024}, Appendix \ref{appendix})
	\begin{align}
	\langle \DD_B \log B.[B \, D + D \, B], D \rangle \ge c^+ \cdot \norm{D}^2.
	\end{align}
Looking back at \eqref{eq1.18} we see that the tangent stiffness tensor $\mathbb{D}(\tau)$ of the Abaqus\textsuperscript{\texttrademark} input format given by
\begin{equation}
	\mathbb{D}(\tau) \coloneqq \frac{1}{J} \, \H^{\ZJ}_{\tau}(\tau)
\end{equation}
is major symmetric and throughout positive definite, such that the Abaqus\textsuperscript{\texttrademark} and Ansys\textsuperscript{\texttrademark} "stability check"
is satisfied globally while the Hencky elasticity model shows unphysical response (e.g.~for purely volumetric response). Thus, the Abaqus\textsuperscript{\texttrademark} and Ansys\textsuperscript{\texttrademark} "stability check" must be used with great care when applied to the compressible case. It does neither imply that the elasticity model is physically sound nor that the FEM-calculations are stable and yielding a mesh - independent result.

\subsection{The incompressible Hencky energy}
In the incompressible case, $\det F = 1$, and we observe that due to
\begin{equation}
	0 = \log 1 = \log \det F = \log \det V = \tr (\log V) = \frac{1}{2} \tr (\log B)
\end{equation}
we have the "linear incompressibility contraint" in the logarithmic strain \quad $\tr (\log V) = 0$.\\
The energy now reads
\begin{equation}
	\WW_{\text{Hencky}}^{\text{inc}} (F) = \widehat{\WW}_{\text{Hencky}}^{\text{inc}} (\log V) = \mu \, \Vert \log V \Vert^2 \, .
\end{equation}
Repeating the calculations for the compressible case, we obtain
\begin{equation}
	\mathbb{D}_{\text{Hencky}}^{\text{inc}} (\tau) = \H_{\tau}^{\ZJ}(\tau)
\end{equation}
with
\begin{equation}
	\H_{\tau}^{\ZJ} (\tau). \, D = \mu \, \DD_B \log B. \, [BD + DB] \quad \quad \text{for} \quad \tr (D) = 0
\end{equation}
such that $\mathbb{D}_{\text{Hencky}}^{\text{inc}} (\tau)$ is positive definite.\footnote
{
	Abaqus\textsuperscript{\texttrademark} (cf.~\cite{abaqus}) is checking the so called ``Drucker-stability'' (cf.~\cite{abaqusTheory}) condition (nothing else than Hill's inequality \\ $\langle \tau (\log V_1) - \tau (\log V_2) \, , \, \log V_1 - \log V_2 \rangle > 0$) in terms of the Kirchhoff stress $\tau$ versus logarithmic strain $\log V$ for different deformation modes and performs a check on the stability of the material for six different forms of loading - uniaxial tension and compression, equibiaxial tension and compression and planar tension and compression - in a stretch range of $0.1 \le \lambda \le 10.0$ at intervals $\Delta \lambda = 0.01$. If an ``instability'' is found, Abaqus\textsuperscript{\texttrademark} issues a warning message and prints the lowest absolute value of $\varepsilon = \log V$ for which the instability is observed.
	In fact Abaqus\textsuperscript{\texttrademark} is using $\mathbb{D}$ as the input for the consistent Jacobian matrix $\texttt{D \!\!\!\!\! D \!\!\!\!\! S \!\!\!\!\! D \!\!\!\!\! D \!\!\!\!\! E}$ (cf.~\cite[eq.(50)]{palizi2020consistent}). The Abaqus\textsuperscript{\texttrademark} stability check is then $\sym \, \mathbb{D} = \mathbb{D} \in \Sym_4^{++}(6)$.
}

\section*{Funding}
This research was supported in part by the Natural Sciences and Engineering Research Council of Canada, through the NSERC Discovery Programme, grant number RGPIN-2024-04247 [S.~Federico].

\begingroup
\footnotesize
%
%
%
\printbibliography
\endgroup
\begin{appendix}

\section{Appendix}
\label{appendix}
\subsection{Notation} \label{appendixnotation}
\textbf{The deformation $\varphi(x,t)$, the material time derivative $\frac{\DD}{\DD t}$ and the partial time derivative $\partial_t$} \\
\\
In accordance with \cite{Marsden1983a} we agree on the following convention regarding an elastic deformation $\varphi$ and time derivatives of material quantities:

Given two sets $\Omega, \Omega_{\xi} \subset \R^3$ we denote by $\varphi\colon\Omega \to \Omega_{\xi}, x \mapsto \varphi(x) = \xi$ the deformation from the \emph{reference configuration} $\Omega$ to the \emph{current configuration} $\Omega_{\xi}$. A \emph{motion} of $\Omega$ is a time-dependent family of deformations, written $\xi = \varphi(x,t)$. The \emph{velocity} of the point $x \in \Omega$ is defined by $\overline{V}(x,t) = \partial_t \varphi(x,t)$ and describes a vector emanating from the point $\xi = \varphi(x,t)$ (see also Figure \ref{yfig1}). Similarly, the velocity viewed as a function of $\xi \in \Omega_{\xi}$ is denoted by $v(\xi,t)$. 

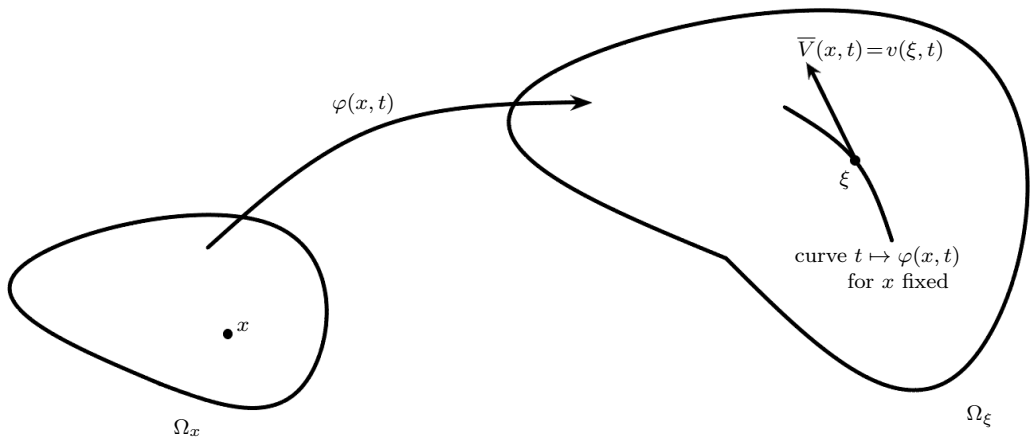
\begin{figure}[h!]
	\begin{center}		
		\begin{minipage}[h!]{0.8\linewidth}
			\centering
			\hspace*{-40pt}
			\includegraphics[scale=0.4]{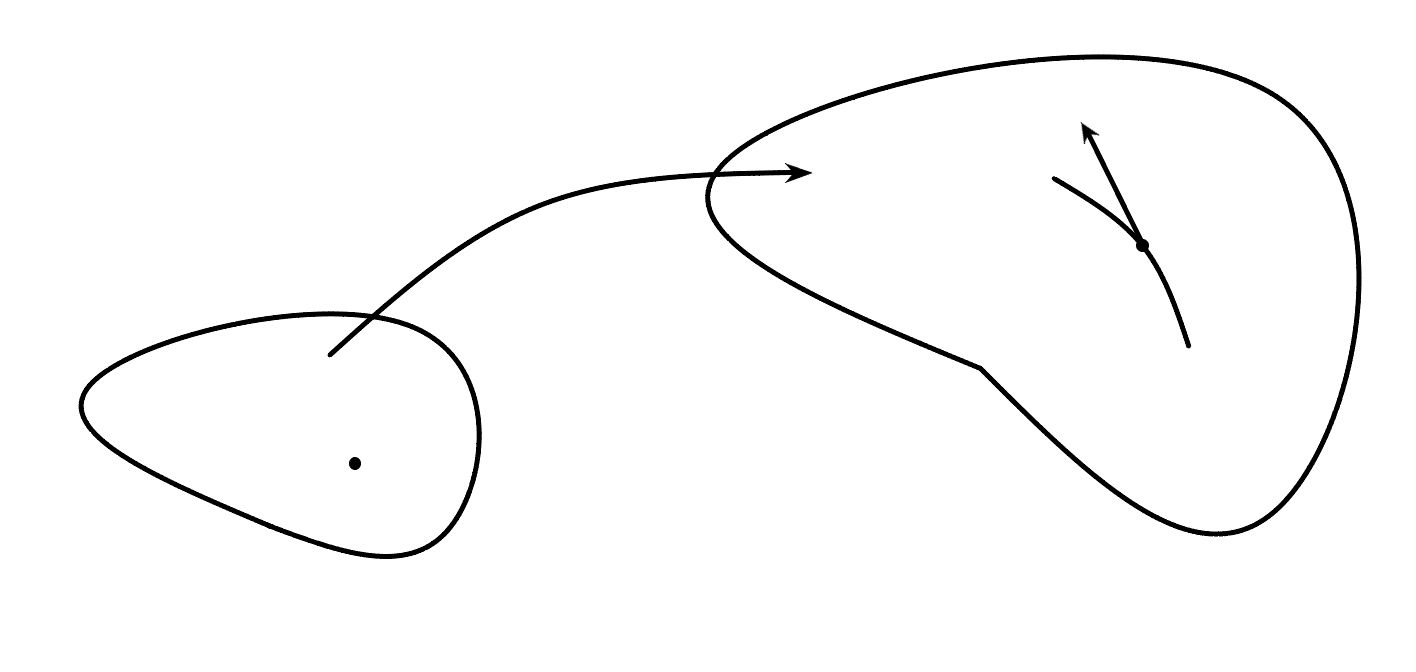}
			\put(-40,30){\footnotesize $\Omega_\xi$}
			\put(-340,25){\footnotesize $\Omega_x$}
			\put(-316,64){\footnotesize $x$}
			\put(-280,148){\footnotesize $\varphi(x,t)$}
			\put(-104,168){\footnotesize $\overline V(x,t) \!=\! v(\xi,t)$}
			\put(-88,119){\footnotesize $\xi$}
			\put(-105,90){\footnotesize curve $t \mapsto \varphi(x,t)$}
			\put(-85,80){\footnotesize  for $x$ fixed}
		\end{minipage} 
		\caption{Illustration of the deformation $\varphi(x,t): \Omega_x \to \Omega_{\xi}$ and the velocity $\overline V(x,t) = v(\xi,t)$.}
		\label{yfig1}
	\end{center}
\end{figure}

Considering an arbitrary material quantity $Q(x,t)$ on $\Omega$, equivalently represented by $q(\xi,t)$ on $\Omega_\xi$, we obtain by the chain rule for the time derivative of $Q(x,t)$
\begin{align}
	\frac{\DD}{\DD t}q(\xi,t) \colonequals \frac{\dif}{\dif t}[Q(x,t)] = \DD_\xi q(\xi,t).v + \partial_t q(\xi,t) \, .
\end{align}
Since it is always possible to view any material quantity $Q(x,t) = q(\xi,t)$ from two different angles, namely by holding $x$ or $\xi$ fixed, we agree to write
\begin{itemize}
	\item $\dot q \colonequals \dd \frac{\DD}{\DD t}[q]$ for the material (substantial) derivative of $q$ with respect to $t$ holding $x$ fixed and
	\item $\partial_t q$ for the derivative of $q$ with respect to $t$ holding $\xi$ fixed.
\end{itemize}
For example, we obtain the velocity gradient $L := \DD_\xi v(\xi,t)$ by
\begin{align}
	L = \DD_\xi v(\xi,t) = \DD_\xi \overline V(x,t) \overset{\text{def}}&{=} \DD_\xi \frac{\dif}{\dif t} \varphi(x,t) = \DD_{\xi} \partial_t \varphi(\varphi^{-1}(\xi,t),t) = \partial_t \DD_x \varphi(\varphi^{-1}(\xi,t),t) \, \DD_\xi \big(\varphi^{-1}(\xi,t)\big) \notag \\
	&=  \partial_t \DD_x \varphi(\varphi^{-1}(\xi,t),t) \, (\DD_x \varphi)^{-1}(\varphi^{-1}(\xi,t),t) = \dot F(x,t) \, F^{-1}(x,t) = L \, ,
\end{align}
where we used that $\partial_t = \frac{\dif}{\dif t} = \frac{\DD}{\DD t}$ are all the same, if $x$ is fixed. \\
\\
As another example, when determining a corotational rate $\frac{\DD^{\circ}}{\DD t}$ we write
\begin{align}
	\frac{\DD^{\circ}}{\DD t}[\sigma] = \frac{\DD}{\DD t}[\sigma] + \sigma \, \Omega^{\circ} - \Omega^{\circ} \, \sigma = \dot \sigma + \sigma \, \Omega^{\circ} - \Omega^{\circ} \, \sigma \, .
\end{align}
However, if we solely work on the current configuration, i.e.~holding $\xi$ fixed, we write $\partial_t v$ for the time-derivative of the velocity (or any quantity in general). \\
\\
\noindent \textbf{Inner product} \\
\\
For $a,b\in\R^n$ we let $\langle {a},{b}\rangle_{\R^n}$  denote the scalar product on $\R^n$ with associated vector norm $\norm{a}_{\R^n}^2=\langle {a},{a}\rangle_{\R^n}$. We denote by $\R^{n\times n}$ the set of real $n\times n$ second-order tensors, written with capital letters. The standard Euclidean scalar product on $\R^{n\times n}$ is given by
$\langle {X},{Y}\rangle_{\R^{n\times n}}=\tr{(X Y^T)}$, where the superscript $^T$ is used to denote transposition. Thus the Frobenius tensor norm is $\norm{X}^2=\langle {X},{X}\rangle_{\R^{n\times n}}$, where we usually omit the subscript $\R^{n\times n}$ in writing the Frobenius tensor norm. The identity tensor on $\R^{n\times n}$ will be denoted by $\id$, so that $\tr{(X)}=\langle {X},{\id}\rangle$. \\
\\
\noindent \textbf{Frequently used spaces} 
\begin{itemize}
	\item $\Sym(n), \rm \Sym^+(n)$ and $\Sym^{++}(n)$ denote the symmetric, positive semi-definite symmetric and positive definite symmetric second-order tensors respectively.
	\item ${\rm GL}(n)\colonequals\{X\in\R^{n\times n}\;|\det{X}\neq 0\}$ denotes the general linear group.
	\item ${\rm GL}^+(n)\colonequals\{X\in\R^{n\times n}\;|\det{X}>0\}$ is the group of invertible matrices with positive determinant.
	\item $\mathrm{O}(n)\colonequals\{X\in {\rm GL}(n)\;|\;X^TX=\id\}$.
	\item ${\rm SO}(n)\colonequals\{X\in {\rm GL}(n,\R)\;|\; X^T X=\id,\;\det{X}=1\}$.
	\item $\mathfrak{so}(3)\colonequals\{X\in\mathbb{R}^{3\times3}\;|\;X^T=-X\}$ is the Lie-algebra of skew symmetric tensors.
	\item The set of positive real numbers is denoted by $\R_+\colonequals(0,\infty)$, while $\overline{\R}_+=\R_+\cup \{\infty\}$.
\end{itemize}
\textbf{Frequently used tensors}
\begin{itemize}
	\item $F = \DD \varphi(x,t)$ is the Fréchet derivative (Jacobean) of the deformation $\varphi(\,,t)\colon\Omega_x \to \Omega_{\xi} \subset \R^3$. $\varphi(x,t)$ is usually assumed to be a diffeomorphism at every time $t \ge 0$ so that the inverse mapping $\varphi^{-1}(\,,t)\colon\Omega_{\xi} \to \Omega_x$ exists.
	\item $C=F^T \, F$ is the right Cauchy-Green strain tensor.
	\item $B=F\, F^T$ is the left Cauchy-Green (or Finger) strain tensor.
	\item $U = \sqrt{F^T \, F} \in \Sym^{++}(3)$ is the right stretch tensor, i.e.~the unique element of ${\rm Sym}^{++}(3)$ with $U^2=C$.
	\item $V = \sqrt{F \, F^T} \in \Sym^{++}(3)$ is the left stretch tensor, i.e.~the unique element of ${\rm Sym}^{++}(3)$ with $V^2=B$.
	\item $\log V = \frac12 \, \log B$ is the spatial logarithmic strain tensor or Hencky strain.
	\item $L = \dot F \, F^{-1} = \DD_\xi v(\xi)$ is the spatial velocity gradient.
	\item $v = \frac{\DD}{\DD t} \varphi(x, t)$ denotes the Eulerian velocity.
	\item $D = \sym \, L$ is the spatial rate of deformation, the Eulerian strain rate tensor.
	\item $W = \sk \, L$ is the vorticity or spin tensor.
	\item We also have the polar decomposition $F = R \, U = V R \in {\rm GL}^+(3)$ with an orthogonal matrix $R \in \OO(3)$ (cf. Neff et al.~\cite{Neffpolardecomp}), see also \cite{LankeitNeffNakatsukasa,Neff_Nagatsukasa_logpolar13}.
\end{itemize}
\noindent \textbf{Frequently used rates}
\begin{multicols}{2}
	\begin{itemize}
		\item $\dd \frac{\DD^{\sharp}}{\DD t}$ denotes an arbitrary objective derivative,
		\item $\dd \frac{\DD^{\circ}}{\DD t} \begin{array}{l} \text{denotes an arbitrary corotational} \\ \text{derivative,} \end{array}$
		\item $\dd \frac{\DD^{\ZJ}}{\DD t} \begin{array}{l} \text{denotes the Zaremba-Jaumann} \\ \text{derivative,} \end{array}$
		\item $\dd \frac{\DD^{\GN}}{\DD t}$ denotes the Green-Naghdi derivative.
		\item $\dd \frac{\DD^{\log}}{\DD t}$ denotes the logarithmic derivative.
		\item $\dd \frac{\DD}{\DD t}$ denotes the material derivative.
	\end{itemize}
\end{multicols}
\noindent \textbf{Tensor domains} \\
\\
Denoting the reference configuration by $\Omega_x$ with tangential space $T_x \Omega_x$ and the current/spatial configuration by $\Omega_\xi$ with tangential space $T_\xi \Omega_\xi$ as well as $\varphi(x) = \xi$, we have the following relations (see also Figure \ref{yfig2}):

\begin{figure}[h!]
	\begin{center}		
		\begin{minipage}[h!]{0.8\linewidth}
			\centering
			\hspace*{-80pt}
			\includegraphics[scale=0.5]{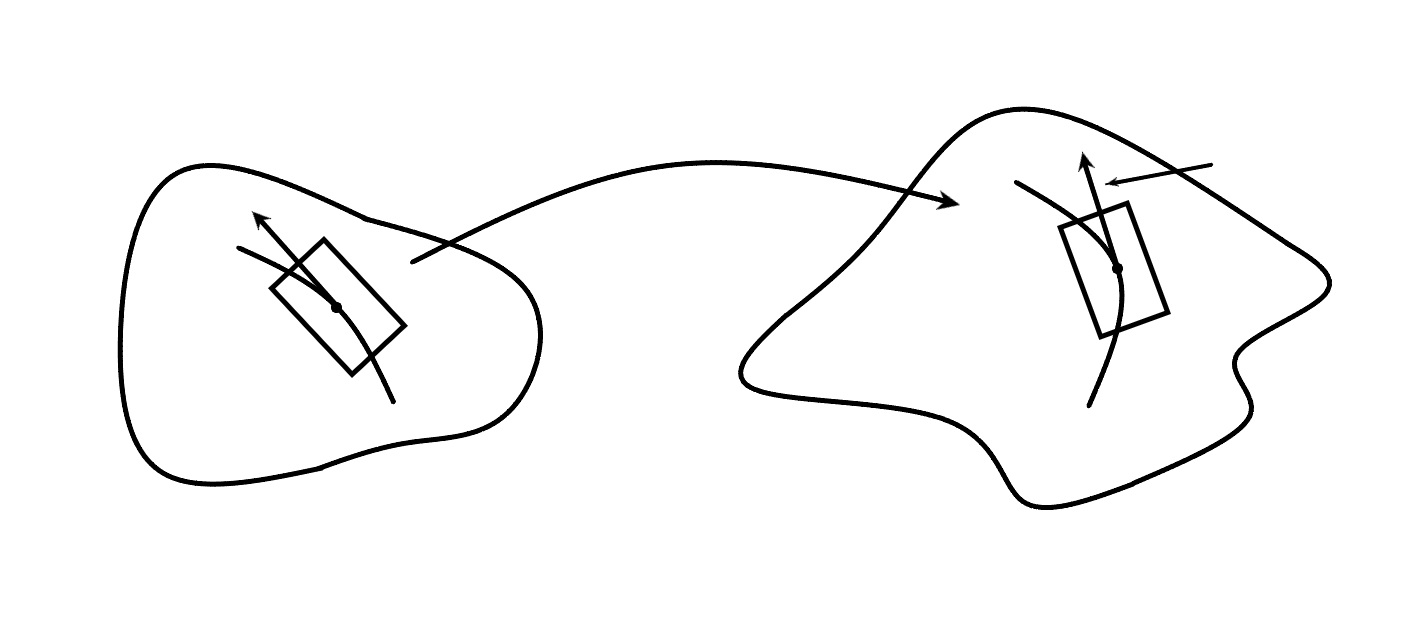}
			\put(-100,45){\footnotesize $\Omega_\xi$}
			\put(-390,55){\footnotesize $\Omega_x$}
			\put(-412,115){\footnotesize $x$}
			\put(-430,155){\footnotesize $\dot \gamma(0)$}
			\put(-377,105){\footnotesize $T_x \Omega_x$}
			\put(-383,85){\footnotesize $\gamma(s)$}
			\put(-280,183){\footnotesize $\varphi(x,t_0)$}
			\put(-120,131){\footnotesize $\xi$}
			\put(-78,181){\footnotesize $\frac{\dif}{\dif s}\varphi(\gamma(s),t_0)\bigg\vert_{s=0}$}
			\put(-88,119){\footnotesize $T_\xi \Omega_\xi$}
			\put(-115,88){\footnotesize $\varphi(\gamma(s),t_0)$}
		\end{minipage} 
		\caption{Illustration of the curve $s \mapsto \varphi(\gamma(s),t_0), \; \gamma(0) = x$ for a fixed time $t = t_0$ with vector field \break $s \mapsto \frac{\dif}{\dif s} \varphi(\gamma(s),t) \in T_\xi \Omega_\xi$.}
		\label{yfig2}
	\end{center}
\end{figure}
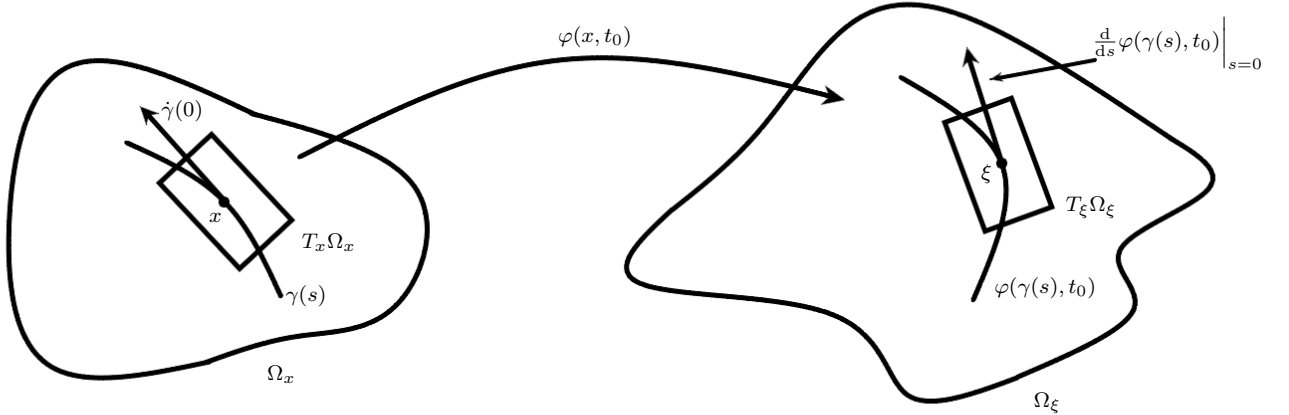
\begin{multicols}{3}
	\begin{itemize}
		\item $F\colon T_x \Omega_x \to T_\xi \Omega_\xi$
		\item $R\colon T_x \Omega_x \to T_\xi \Omega \xi$
		\item $F^T\colon T_\xi \Omega_\xi \to T_x \Omega_x$
		\item $R^T\colon T_\xi \Omega_\xi \to T_x \Omega_x$
		\item $C = F^T \, F\colon T_x \Omega_x \to T_x \Omega_x$
		\item $B = F \, F^T\colon T_\xi \Omega_\xi \to T_\xi \Omega_\xi$
		\item $\sigma\colon T_\xi \Omega_\xi \to T_\xi \Omega_\xi$
		\item $\tau\colon T_\xi \Omega_\xi \to T_\xi \Omega_\xi$
		\item $S_2 \colon T_x \Omega_x \to T_x \Omega_x$
		\item $S_1 \colon T_x \Omega_x \to T_\xi \Omega_\xi$
		\item $R^T \, \sigma \, R\colon T_x \Omega_x \to T_x \Omega_x$
	\end{itemize}
\end{multicols}
\noindent \textbf{The strain energy function $\WW(F)$} \\
\\
We are only concerned with rotationally symmetric functions $\WW(F)$ (objective and isotropic), i.e.
\begin{equation*}
	\WW(F)={\WW}(Q_1^T\, F\, Q_2), \qquad \forall \, F \in {\rm GL}^+(3), \qquad  Q_1 , Q_2 \in {\rm SO}(3).
\end{equation*}
\textbf{List of additional definitions and useful identities}
\begin{itemize}
	\item For two metric spaces $X, Y$ and a linear map $L\colon X \to Y$ with argument $v \in X$ we write $L.v\colonequals L(v)$. This applies to a second-order tensor $A$ and a vector $v$ as $A.v$ as well as a fourth-order tensor $\C$ and a second-order tensor $H$ as $\C.H$. Sometimes we may emphasise the usual matrix product of two second-order tensors $A, B$ as $A \cdot B$.
	\item We define $J = \det{F}$ and denote by $\Cof \, X = (\det X)X^{-T}$ the cofactor of a matrix in ${\rm GL}^{+}(3)$.
	\item We define $\sym X = \frac12 \, (X + X^T)$ and $\sk X = \frac12 \, (X - X^T)$ as well as $\dev X = X - \frac13 \, \tr(X) \, \id$.
	\item For all vectors $\xi,\eta\in\R^3$ we have the tensor product $(\xi\otimes\eta)_{ij}=\xi_i\,\eta_j$.
	\item $S_1=\DD_F \WW(F) = \sigma \, \Cof F$ is the non-symmetric first Piola-Kirchhoff stress tensor.
	\item $S_2=F^{-1}S_1=2\,\DD_C \widetilde{\WW}(C)$ is the symmetric second  Piola-Kirchhoff stress tensor.
	\item $\sigma=\frac{1}{J}\,  S_1\, F^T=\frac{1}{J}\,  F\,S_2\, F^T=\frac{2}{J}\DD_B \widetilde{\WW}(B)\, B=\frac{1}{J}\DD_V \widetilde{\WW}(V)\, V = \frac{1}{J} \, \DD_{\log V} \widehat \WW(\log V)$ is the symmetric Cauchy stress tensor.
	\item $\sigma = \frac{1}{J} \, F\, S_2 \, F^T = \frac{2}{J} \, F \, \DD_C \widetilde{\WW}(C) \, F^T$ is the ``\emph{Doyle-Ericksen formula}'' \cite{doyle1956}.
	\item For $\sigma\colon \Sym(3) \to \Sym(3)$ we denote by $\DD_B \sigma(B)$ with $\sigma(B+H) = \sigma(B) + \DD_B \sigma(B).H + o(H)$ the Fréchet-derivative. For $\sigma\colon \Sym^+(3) \subset \Sym(3) \to \Sym(3)$ the same applies. Similarly, for $\WW\colon\R^{3 \times 3} \to \R$ we have $\WW(X + H) = \WW(X) + \langle \DD_X \WW(X), H \rangle + o(H)$.
	\item $\tau = J \, \sigma = 2\, \DD_B \widetilde{\WW}(B)\, B $ is the symmetric Kirchhoff stress tensor.
	\item $\tau = \DD_{\log V} \widehat{\WW}(\log V)$ is the ``\emph{Richter-formula}'' \cite{richter1948isotrope, richter1949hauptaufsatze}.
	\item $\sigma_i =\dd\frac{1}{\lambda_1\lambda_2\lambda_3}\dd\lambda_i\frac{\partial g(\lambda_1,\lambda_2,\lambda_3)}{\partial \lambda_i}=\dd\frac{1}{\lambda_j\lambda_k}\dd\frac{\partial g(\lambda_1,\lambda_2,\lambda_3)}{\partial \lambda_i}, \ \ i\neq j\neq k \neq i$ are the principal Cauchy stresses (the eigenvalues of the Cauchy stress tensor $\sigma$), where $g:\mathbb{R}_+^3\to \mathbb{R}$ is the unique function  of the singular values of $U$ (the principal stretches) such that $\WW(F)=\widetilde{\WW}(U)=g(\lambda_1,\lambda_2,\lambda_3)$.
	\item $\sigma_i =\dd\frac{1}{\lambda_1\lambda_2\lambda_3}\frac{\partial \widehat{g}(\log \lambda_1,\log \lambda_2,\log \lambda_3)}{\partial \log \lambda_i}$, where $\widehat{g}:\mathbb{R}^3\to \mathbb{R}$ is the unique function such that \\ \hspace*{0.3cm} $\widehat{g}(\log \lambda_1,\log \lambda_2,\log \lambda_3)\colonequals g(\lambda_1,\lambda_2,\lambda_3)$.
	\item $\tau_i =J\, \sigma_i=\dd\lambda_i\frac{\partial g(\lambda_1,\lambda_2,\lambda_3)}{\partial \lambda_i}=\frac{\partial \widehat{g}(\log \lambda_1,\log \lambda_2,\log \lambda_3)}{\partial \log \lambda_i}$ \, . 
\end{itemize}

\vspace*{2em}
\noindent \textbf{Conventions for fourth-order symmetric operators, minor and major symmetry} \\
\\
Fourth order tensors are written as $\H$ or $\C$. For a fourth-order linear mapping $\C\colon\Sym(3) \to \Sym(3)$ we agree on the following convention. \\
\\
We say that $\C$ has \emph{minor symmetry} if
\begin{align}
	\C.S \in \Sym(3) \qquad \forall \, S \in \Sym(3).
\end{align}
This can also be written in index notation as $C_{ijkm} = C_{jikm} = C_{ijmk}$. If we consider a more general fourth-order tensor $\C\colon\R^{3 \times 3} \to \R^{3 \times 3}$ then $\C$ can be transformed having minor symmetry by considering the mapping $X \mapsto \sym(\C. \sym X)$ such that $\C\colon\R^{3 \times 3} \to \R^{3 \times 3}$ is minor symmetric, if and only if $\C.X = \sym(\C.\sym X)$. \\
\\
We say that $\C$ has \emph{major symmetry} (or is \emph{self-adjoint}, respectively) if
\begin{align}
	\langle \C. S_1, S_2 \rangle = \langle \C. S_2, S_1 \rangle \qquad \forall \, S_1, S_2 \in \Sym(3).
\end{align}
Major symmetry in index notation is understood as $C_{ijkm} = C_{kmij}$. \\
\\
The set of positive definite, major symmetric fourth-order tensors mapping $\R^{3 \times 3} \to \R^{3 \times 3}$ is denoted as $\Sym^{++}_4(9)$, in case of additional minor symmetry, i.e.~mapping $\Sym(3) \to \Sym(3)$ as $\Sym^{++}_4(6)$. By identifying $\Sym(3) \cong \R^6$, we can view $\C$ as a linear mapping in matrix form $\widetilde \C\colon\R^6 \to \R^6$. \newline If $H \in \Sym(3) \cong \R^6$ has the entries $H_{ij}$, we can write
\begin{align}
	\label{eqvec1}
	h = \textnormal{vec}(H) = (H_{11}, H_{22}, H_{33}, H_{12}, H_{23}, H_{31}) \in \R^6 \qquad \textnormal{so that} \qquad \langle \C.H, H \rangle_{\Sym(3)} = \langle \widetilde \C.h, h \rangle_{\R^6}.
\end{align}
If $\C\colon\Sym(3) \to \Sym(3)$, we can define $\bfsym \C$ by
\begin{align}
	\langle \C.H, H \rangle_{\Sym(3)} = \langle \widetilde \C.h, h \rangle_{\R^6} = \langle \sym \, \widetilde \C. h, h \rangle_{\R^6} \equalscolon \langle \bfsym \C.H, H \rangle_{\Sym(3)}, \qquad \forall \, H \in \Sym(3).
\end{align}
Major symmetry in these terms can be expressed as $\widetilde \C \in \Sym(6)$. \emph{In this text, however, we omit the tilde-operation and ${\bf sym}$ and write in short $\sym\,\C\in {\rm Sym}_4(6)$ if no confusion can arise.} In the same manner we speak about $\det \C$ meaning $\det \widetilde \C$. \\
\\
A linear mapping $\C\colon\R^{3 \times 3} \to \R^{3 \times 3}$ is positive definite if and only if
\begin{align}
	\label{eqposdef1}
	\langle \C.H, H \rangle > 0 \qquad \forall \, H \in \R^{3 \times 3} \qquad \iff \qquad \C \in \Sym^{++}_4(9)
\end{align}
and analogously it is positive semi-definite if and only if
\begin{align}
	\label{eqpossemidef1}
	\langle \C.H, H \rangle \ge 0 \qquad \forall \, H \in \R^{3 \times 3} \qquad \iff \qquad \C \in \Sym^+_4(9).
\end{align}
For $\C\colon\Sym(3) \to \Sym(3)$, after identifying $\Sym(3) \cong \R^6$, we can reformulate \eqref{eqposdef1} as $\widetilde \C \in \Sym^{++}(6)$ and \eqref{eqpossemidef1} as $\widetilde \C \in \Sym^+(6)$.\\
\subsection{Some objective derivatives}
\label{sec:Some_objective_derivatives}
Typical representatives for objective derivatives $\frac{\DD^{\sharp}}{\DD t}$ are
\begin{alignat}{2}
	\frac{\DD^{\CR}}{\DD t}[\sigma] &:= &&\; \frac{\DD}{\DD t}[\sigma] + L^T \, \sigma + \sigma \, L \quad \text{(non-corotational Cotter-Rivlin derivative (cf.~\cite{Cotter1955TENSORSAW})).} \notag \\
	\frac{\DD^{\Old}}{\DD t}[\sigma] &:= &&\; \frac{\DD}{\DD t}[\sigma] - (L \, \sigma + \sigma \, L^T) \quad \text{(non-corotational convective contravariant Oldroyd derivative (cf.~\cite{oldroyd1950}))}, \notag \\
	\frac{\DD^{\text{Hencky}}}{\DD t}[\sigma] &:= &&\; \frac{\DD}{\DD t}[\sigma] + \sigma \, W - W \, \sigma + \sigma \, \tr(D) \quad \text{(non-corotational Biezeno-Hencky derivative (cf.~\cite{biezeno1928})}, \notag \\
	\label{eqobjrates}
	& &&\hspace{5.1cm} \text{sometimes also called Hill-rate (cf.~\cite{korobeynikov2023})),} \\
	\frac{\DD^{\TR}}{\DD t}[\sigma] &:= &&\; \frac{\DD}{\DD t}[\sigma] - (L \, \sigma + \sigma \, L^T) + \sigma \, \tr(D) \quad \text{(non-corotational Truesdell derivative (cf.~\cite[eq.~3]{truesdellremarks})).} \notag
\end{alignat}
Even though it will not play a role in our development, we would like to point out that there is an intimate relation of these objective rates to Lie derivatives and covariant derivatives (cf.~\cite{kolev2024objective, Marsden1983a}).

The most well-known members of the subfamily of corotational derivatives (cf. \cite{pos_cor}) are
	\begin{alignat}{2}
		\frac{\DD^{\ZJ}}{\DD t}[\sigma] &:= && \; \frac{\DD}{\DD t}[\sigma] + \sigma \, W - W \, \sigma = Q^W \, \frac{\DD}{\DD t}[(Q^W)^T \, \sigma \, Q^W] \, (Q^W)^T, \quad  \text{for $Q^W(t) \in \OO(3)$ with} \; W = \dot{Q}^W \, (Q^W)^T, \notag \\
		& &&\; \text{where} \; W = \sk L \;\; \text{is the vorticity} \quad (\text{corotational \textbf{Zaremba-Jaumann derivative} (cf.~\cite{jaumann1905, jaumann1911geschlossenes, zaremba1903forme}}), \notag \\
		\frac{\DD^{\GN}}{\DD t}[\sigma] &:= &&\; \frac{\DD}{\DD t}[\sigma] + \sigma \, \Omega^R - \Omega^R \, \sigma = R \, \frac{\DD}{\DD t}[R^T \, \sigma \, R] \, R^T, \quad \text{for $R(t) \in \OO(3)$ with the ``polar spin''} \, \Omega^R := \dot{R} \, R^T, \notag \\
		& &&\; \text{with the polar decomposition} \; F = R \, U \quad \text{(corotational \textbf{Green-Naghdi derivative} (cf.~\cite{bellini2015, Green_McInnis_1967,  Green1965, naghdi1961}))}, \notag \\
		\frac{\DD^{\log}}{\DD t}[\sigma] &:= &&\; \frac{\DD}{\DD t}[\sigma] + \sigma \, \Omega^{\log} - \Omega^{\log} \, \sigma, \quad \text{for $Q^{\log}(t) \in \OO(3)$ with the ``logarithmic spin''} \; \Omega^{\log} = \dot{Q}^{\log} \, (Q^{\log})^T \notag \\
		& &&\; \text{(corotational \textbf{logarithmic derivative} (cf.~\cite{xiao98_1}))}.
	\end{alignat}
\subsection{Uniaxial stretch - one dimensional insights}
We find it illuminating to regard the rate-formulation in a uniaxial situation, in which the Zaremba-Jaumann derivative $\frac{\DD^{\ZJ}}{\DD t}$ (and indeed any corotational derivative $\frac{\DD^{\circ}}{\DD t}$ (cf.\,\cite{pos_cor})) reduces to the material derivative $\frac{\DD}{\DD t}$ since rotation effects are excluded. More precisely, we consider the deformation
\begin{equation}
	\phi (x_1,x_2,x_3) = (\lambda \, x_1, x_2, x_3)
\end{equation}
where $\lambda > 0$ is the uniaxial stretch. Then we have
\begin{equation}
	\begin{alignedat}{2}
		F&= \diag (\lambda, 1, 1) ,\quad F^T = \diag (\lambda, 1 , 1) , \quad J = \det F = \lambda, \\
		\Cof F &= \det F \, F^{-T} = \lambda \, \diag (\frac{1}{\lambda}, 1, 1)
		= \diag (1, \lambda, \lambda)
		\, ,\\
		\WW(\lambda):&= \WW \left( \diag (\lambda, 1, 1) \right) = \WW (F) \, ,\\
		W &= \skw L = \skw (\dot{F} F^{-1}) = \skw \big( \diag (\dot{\lambda}, 1, 1) \cdot \diag (\lambda, 1, 1)^{-1}\big) = \skw \diag \big( \frac{\dot{\lambda}}{\lambda}, 1, 1\big) = 0\, ,\\
		S_1 (F) &= \DD_F \WW (F) \big\vert_{F = \text{diag}(\lambda, 1, 1)} = \diag (\DD_{\lambda} \WW (\lambda), 0, 0) \, .
	\end{alignedat}
\end{equation}
%

\begingroup
\renewcommand\arraystretch{2}
\begin{longtable}[h!]{| c | c |}
	\hline
	\multicolumn{2}{| c |}{stress-tensors}\\
	\hline
	\textbf{1D-simplification} & \textbf{3D ideal isotropic nonlinear elasticity} \\
	\hline
	\endfirsthead
	\hline
	\multicolumn{2}{|c|}{stress-tensors}\\
	\hline
	\textbf{1D} & \textbf{3D ideal isotropic nonlinear elasticity} \\
	\hline
	\endhead
	\hline
	\endfoot
	\hline
	\endlastfoot
	\makecell{energy $\lambda \mapsto \WW(\lambda) = \widehat{\WW}(\log \lambda) = \widetilde{W}(\lambda^2)$} & \makecell{$F \mapsto \WW(F) = \WW(V) = \widehat{\WW}(\log V) = \widetilde{W} (C)$} \\
	\hline
	\makecell{Cauchy stress \\ $\sigma(\lambda) = \frac{1}{\lambda} \, \DD_{\lambda} \, \WW (\lambda) \cdot \lambda$ , \\  $\DD_{\lambda}\WW(\lambda) = \widehat{\sigma}(\log \lambda)$} & \makecell{Cauchy stress \\ $\sigma = \frac{1}{J} \, S_1 \, F^T, \quad J = \det F$} \\
	\hline
	\makecell{Cauchy stress \\ $\sigma (\lambda) = \frac{1}{\lambda} \, \left(\lambda \, 2 \, \DD_{\lambda^2} \widetilde{\WW}(\lambda^2) \, \lambda \right)$} & \makecell{Cauchy stress \\ $\sigma = \frac{1}{J} \, F \, S_2 \, F^T$ \quad ``Doyle-Ericksen formula''} \\
	\hline
	\makecell{1. Piola-Kirchhoff stress \\ $\sigma(\lambda) = \DD_{\lambda}\WW(\lambda)$ \\ $\sigma (\lambda) = \lambda \cdot 2 \, \DD_{\lambda^2} \widetilde{W} (\lambda^2) = \DD_{\lambda} \WW (\lambda) $} & \makecell{1. Piola-Kirchhoff stress \\ $S_1(F) = \DD_F\WW(F) \ = F \cdot S_2$ } \\
	\hline
	\makecell{Biot-stress \\ $\sigma(\lambda) = \DD_{\lambda}\WW(\lambda)$} & \makecell{Biot-stress \\ $T_{\text{Biot}} = \DD_U\WW(U)$} \\
	\hline
	\makecell{2. Piola-Kirchhoff stress \\ $S_2(\lambda) = 2 \, \DD_{\lambda^2} \widetilde{\WW}(\lambda^2)$} & \makecell{2. Piola-Kirchhoff stress \\ $S_2(C) = 2 \, \DD_C \widetilde{\WW}(C)$} \\
	\hline
	\makecell{Kirchhoff stress \\ $\tau (\lambda) = \lambda \cdot \sigma(\lambda) = \DD_{\log \lambda} \widehat{\WW}(\log \lambda)$} & \makecell{Kirchhoff stress (weighted Cauchy stress) \\ $\tau = J \, \sigma = \DD_{\log V} \widehat{\WW}(\log V) = 2 \, \DD_{\log B} \doublehat{\WW}(\log B)$} \\
	\hline
	\makecell{Kirchhoff stress \\ $\tau = \DD_{\log \lambda} \widehat{\WW}(\log \lambda)$} & \makecell{``Richter's formula'' \\ $\tau = \DD_{\log V} \widehat{\WW}(\log V)$} \\
	\hline
	\rule{0.1\textwidth}{1pt}& \makecell{isotropic linear elasticity \\ $\sigma = \C^{\iso}. \, \varepsilon = 2\mu  \, \dev \varepsilon + \kappa \, \tr(\varepsilon) \id$ \\ $\C^{\iso}$ invertible \quad $\iff \quad \mu \neq 0, \; \kappa \neq 0$ \\ $\C^{\iso} \in \Sym^{++}_4(6) \quad \iff \quad \mu, \kappa>0$}\\
	\hline
	\makecell{$\frac{\DD^{\ZJ}}{\DD t} = \frac{\DD}{\DD t}$} & \makecell{$\frac{\DD^{\ZJ}}{\DD t} [\sigma] = \frac{\DD}{\DD t} [\sigma] + W\sigma - \sigma W$}
\end{longtable}
\endgroup


We observe that in the uniaxial situation we cannot distinguish between the Biot-stress $T_{\Biot}$ (the engineering stress), the first Piola -Kirchhoff stress $S_1$ and the Cauchy stress $\sigma$. However, it is possible to see the essential differences appearing for the Kirchhoff stress $\tau$, the Cauchy stress $\sigma$ and the second Piola-Kirchhoff stress $S_2$.
\begingroup
\renewcommand\arraystretch{2}
\begin{longtable}[h!]{| c | c |}
	\hline
	\textbf{1D-simplification} & \textbf{3D ideal isotropic nonlinear elasticity} \\
	\hline
	\endfirsthead
	\hline
	\textbf{1D-simplification} & \textbf{3D ideal isotropic nonlinear elasticity} \\
	\hline
	\endhead
	\hline
	\endfoot
	%
	\endlastfoot
	\makecell{Hill's inequality \\ $\WW(\lambda) = \widehat{\WW}(\log \lambda), \quad \log \lambda \mapsto \widehat{\WW}(\log \lambda) \quad$ is convex \\ $\notiff \lambda \mapsto \WW(\lambda)$ is convex: \quad $[\WW(\lambda) = \frac12 (\log \lambda)^2]$} & \makecell{Hill's inequality, hyperelasticity (Sidoroff \cite{sidoroff1974restrictions}) \medskip \\ $\WW(F) = \widehat{\WW}(\log V), \quad \log V \mapsto \widehat{\WW}(\log V)$ \quad is convex} \\
	\hline
	\makecell{Hill's inequality \\ $\textbf{(}\widehat{\tau}(\log \lambda_1) - \widehat{\tau}(\log \lambda_2)\textbf{)}\textbf{(}\log \lambda_1 - \log \lambda_2\textbf{)} > 0$ \\ $\notiff \textbf{\{}\DD_{\lambda} \sigma(\lambda) > 0 \Leftrightarrow \DD_{\log \lambda} \widehat{\sigma} (\log \lambda) > 0\textbf{\}}$} & \makecell{Hill's inequality, Cauchy elastic \\ $\pmb{\langle} \widehat{\tau}(\log V_1) - \widehat{\tau}(\log V_2), \log V_1 - \log V_2 \pmb{\rangle} > 0$ \\ $\notiff \det \DD \sigma(V) > 0$} \\
	\hline
	\makecell{Hill's inequality \\ $\H_{\tau} (\tau) > 0 \quad \iff \quad \DD_{\log \lambda}^2 \, \widehat{\WW} (\log \lambda) > 0$} & \makecell{Hill's inequality \\ $\H_{\tau}(\tau) \in \Sym_4^{++}(6) \quad \iff \quad \DD_{\log V} \, \widehat{\tau} (\log V) \in \Sym_4
		^{++}(6)$}\\
	\hline
	\makecell{TSTS-M$^+$ \\ $\textbf{(}\widehat{\sigma} (\log \lambda_1) - \widehat{\sigma} (\log \lambda_2)\textbf{)} \textbf{(} \log \lambda_1 - \log \lambda_2 \textbf{)} > 0$} & \makecell{TSTS-M$^+$ \\  $\pmb{\langle} \widehat{\sigma} (\log V_1) - \widehat{\sigma} (\log V_2) \, , \, \log V_1 - \log V_2 \pmb{\rangle} > 0$}\\
	\hline
	\makecell{TSTS-M$^{++}$ \\ $\DD_{\log \lambda} \widehat{\sigma} (\log \lambda) > 0$} & \makecell{TSTS-M$^{++}$\\ $\sym \, \DD_{\log V} \widehat{\sigma} (\log V) \in \Sym_4^{++}(6)$}\\
	\hline
	\caption{Here $\H_{\tau} (\tau)$ denotes the induced stiffness tensor in the rate formulation $\frac{\DD^{\ZJ}}{\DD t} [\tau] = \H_{\tau}(\tau) \ldot D$.}
\end{longtable}
\endgroup
\vspace{-2em}
\begingroup
\renewcommand\arraystretch{2}
\begin{longtable}[h!]{| c | c |}
	\hline
	\textbf{1D-simplification} & \textbf{3D ideal isotropic nonlinear elasticity} \\
	\hline
	\endfirsthead
	\hline
	\textbf{1D-simplification} & \textbf{3D ideal isotropic nonlinear elasticity} \\
	\hline
	\endhead
	\hline
	\endfoot
	\hline
	\endlastfoot
	\makecell{$\lambda^2 \mapsto \widetilde{\WW}(\lambda^2), \quad \WW(\lambda) = \widetilde{\WW}(\lambda^2)$ \quad is convex \\ $\notiff \lambda \mapsto \WW(\lambda)$ is convex \quad $[W_{\text{SVK}}(\lambda) = \frac18(\lambda^2-1)^2]$} & \makecell{\rule{0pt}{13pt}$\WW(F) = \widetilde{\WW}(C), \quad C \mapsto \widetilde{\WW}(C)$ \quad is convex, \\ can be consistent with polyconvexity \\ and $\WW(F) \to + \infty$ for $\det F \to 0$, \\ but does not exclude otherwise problematic response} \\
	\hline
	\makecell{\rule{0pt}{21pt}
	{\begin{math}
			\begin{aligned}
				\tau &= \lambda \, \sigma(\lambda) = 2 \, \DD_{\lambda^2}\widetilde{\WW}(\lambda^2) \, \lambda^2
				=\DD_{\log \lambda}\widehat{\WW}(\log \lambda) \\ &= \DD_{\lambda}[\widehat{\WW}(\log \lambda)] \lambda =\lambda \, \DD_{\lambda}\WW(\lambda)
		\end{aligned}
	\end{math}}
	} & \makecell{
	\begin{math}
		\begin{aligned}
			\tau = \det F \cdot \sigma(V) &= \DD_{\log V}\widehat{\WW}(\log V) = 2 \, \DD_B \widetilde{\WW}(B) B \\
			&= \DD_V \widehat{\WW}(\log V) \cdot V
		\end{aligned}
	\end{math}
	} \\
	\hline
	$S_2(\lambda^2) = 2 \, \DD_{\lambda^2}\widetilde{\WW}(\lambda^2)$ & \makecell{$S_2(C) = 2 \, \DD_C \widetilde{\WW}(C) = 2 \, \DD_B \widetilde{\WW}(B)$ \\ $\C:= 4 \, \DD^2_C \widetilde{\WW}(C) = 2 \, \DD_C S_2(C),$ \quad second elasticity tensor} \\
	\hline
	\makecell{$\lambda^2 \mapsto \widetilde{\WW} (\lambda^2)$ convex} & \makecell{$\mathbb{C} \in \Sym_4^{++}(6) \quad \iff \quad C \mapsto \widetilde{\WW} (C)$ convex}\\
	\hline
	\makecell{$\DD_{\lambda}^2 \, \WW (\lambda) > 0 \quad \iff \quad \DD_{\lambda} \sigma (\lambda) > 0$ \\ $ \quad \iff \quad \DD_{\log \lambda} \, \widehat{\sigma} (\log \lambda) > 0$} & \makecell{$\mathbb{A} := \DD_F^2 \, \WW (F) = \DD_F \, S_1 (F),$ \quad first elasticity tensor}\\
	\hline
	\makecell{$\DD_{\lambda}^2 \, \WW (\lambda) > 0 \quad \iff \quad \lambda \mapsto \WW (\lambda)$ convex} & \makecell{$\mathbb{A} \in \Sym_4^{++}(9) \quad \iff \quad F \mapsto \WW (F)$ convex}\\
\end{longtable}
\endgroup
\begingroup
\begin{longtable}[h!]{| c | c |}
	\hline
	\multicolumn{2}{| c |}{rate-type conditions in hypoelasticity}\\
	\hline
	\textbf{1D-simplification} & \textbf{3D ideal isotropic nonlinear elasticity} \\
	\hline
	\endfirsthead
	\hline
	\multicolumn{2}{|c|}{rate type conditions}\\
	\hline
	\textbf{1D-simplification} & \textbf{3D ideal isotropic nonlinear elasticity} \\
	\hline
	\endhead
	\hline
	\endfoot
	\endlastfoot
	$\frac{\dif}{\dif t}[\sigma] = 2 \, \mu \cdot D, \quad D=\frac{\dot{\lambda}}{\lambda} \quad \implies \quad \widehat{\sigma}(\log \lambda) = 2 \, \mu \, \log \lambda$ & \makecell{\vspace*{0.0005em}\\ $\frac{\DD^{\log}}{\DD t}[\sigma] = \C^{\iso} \, . \, D, \quad D= \sym(\dot{F} \, F^{-1})$ \\ $\implies \widehat{\sigma}(\log V) = \C^{\iso} \, . \, \log V$, but not hyperelastic (Bruhns)} \\
	\hline
	\vspace*{0.005em}\\ $\frac{\dif}{\dif t}[\sigma] = 2 \, \mu \cdot D, \quad D=\frac{\dot{\lambda}}{\lambda} \quad \implies \quad \widehat{\sigma}(\log \lambda) = 2 \, \mu \, \log \lambda$ & \makecell{$\frac{\DD^{\ZJ}}{\DD t}[\sigma] = \DD_B \, \sigma(B) . [B \, D + D \, B], \quad D= \sym(\dot{F} \, F^{-1})$ \\
	\vspace*{0.005em}\\} \\
	\hline
	\makecell{\vspace*{0.005em}\\ $\frac{\dif}{\dif t}[\sigma] = \H(\sigma) \, D \, , \quad D = \frac{\dot{\lambda}}{\lambda} \quad \text{with} \quad \H(\sigma) > 0$ \\ $\implies \left\{ \begin{array}{rll} \log \lambda &\mapsto \widehat{\sigma}(\log \lambda) &\quad \text{is monotone} \\ \lambda &\mapsto \sigma(\lambda) &\quad \text{is monotone} \end{array} \right.$} & \makecell{$\frac{\DD^{\ZJ}}{\DD t}[\sigma] = \H^{\ZJ}(\sigma) \, . \, D, \quad \sym \, \H^{\ZJ}(\sigma) \in \Sym^{++}_4(6)$ \\ $\implies \log V \mapsto \widehat{\sigma}(\log V)$ is monotone \\ \textbf{but} $V \mapsto \sigma (V)$ may \textbf{not} be monotone} \\
	\hline
	\rule{0.1\textwidth}{1pt} & \makecell{$\frac{\DD^{\log}}{\DD t}[\sigma] = \H^{\log}(\sigma) \, . \, D, \quad \sym \, \H^{\log}(\sigma) \in \Sym^{++}_4(6)$ \\ $\implies\log V \mapsto \widehat{\sigma}(\log V)$ is monotone} \\
	\hline
	\makecell{\vspace{0.0005em} \\ $\frac{\dif}{\dif t}[\sigma] = \H(\sigma) \, D, \quad D = \frac{\dot{\lambda}}{\lambda},
	\quad \H(\sigma) > 0$ \\ $\log \lambda \mapsto \widehat{\sigma}(\log \lambda)$ is invertible} & \makecell{$\frac{\DD^{\ZJ}}{\DD t}[\sigma] = \H^{\ZJ}(\sigma) \, . \, D, \quad \det \, \H^{\ZJ}(\sigma) > 0$ \\ $\iff V \mapsto \sigma(V)$ bijective} \\
	\hline
	\makecell{\vspace{0.0005em} \\ $\frac{\dif}{\dif t}[\sigma] = \H(\sigma) \, D, \quad D=\frac{\dot{\lambda}}{\lambda}, \quad \quad \H(\sigma) > 0$ \\ $\implies \log \lambda \mapsto \widehat{\sigma}(\log \lambda)$ is monotone} & \makecell{$\frac{\DD^{\log}}{\DD t}[\sigma] = \H^{\log}(\sigma) \, . \, D, \quad \det \, \H^{\log}(\sigma) > 0$ \\ $\iff V \mapsto \sigma(V)$ bijective} \\
	\hline
	\makecell{hyperelasticity: \medskip \\ $\WW(\lambda)$ consistent with $\frac{\dif}{\dif t}[\sigma] = \H(\sigma) \, D$ (present paper)} & \makecell{hyperelasticity: \\ $\WW(F)$ consistent with $\frac{\DD^{\ZJ}}{\DD t}[\sigma] = \H^{\ZJ}(\sigma) \, . \, D$ \\ if $V \mapsto \sigma(V)$ is invertible (Truesdell 1963)} \\
	\hline
	\makecell{\vspace*{0.005em}\\ $\frac{\dif}{\dif t}[\sigma] = \H(\sigma) \, \frac{\dot{\lambda}}{\lambda}$ \\ $\iff \quad \H(\sigma) = \DD_{\log \lambda} \, \widehat{\sigma}(\log \lambda) = \DD_{\lambda} \, \sigma(\lambda) \cdot \lambda$} & \makecell{\vspace*{0.005em}\\{\begin{math}
				\begin{aligned}
					\frac{\DD^{\ZJ}}{\DD t}[\sigma] &= \H^{\ZJ}(\sigma) \, . \, D, \\
					\H^{\ZJ}(\sigma)&= \DD_B \, \sigma(B) \, . \, [B \, D + D \, B] \\
					&= \DD_{\log B} \, \widehat{\sigma} (\log B) \, . \, \DD_B \, \log B \, . \, [B \, D + D \, B]
				\end{aligned}
		\end{math}}
		} \\
	\hline
	\makecell{\vspace*{0.005em} \\ $\frac{\dif}{\dif t}[\sigma] = \H(\sigma) \, \frac{\dot{\lambda}}{\lambda}$ \\ $\iff \quad \H(\sigma) = \DD_{\log \lambda} \, \widehat{\sigma}(\log \lambda) = \DD_{\lambda} \, \sigma(\lambda) \cdot \lambda$} & \makecell{$\frac{\DD^{\log}}{\DD t}[\sigma] = \H^{\log}(\sigma) \, . \, D, \quad \H^{\log}(\sigma)= \DD_{\log V} \, \widehat{\sigma} (\log V)$} \\
	\hline
	\caption{Results on the right hand side are drawn from \cite{CSP2024}.}
\end{longtable}
\endgroup

\subsection{One-dimensional hypoelasticity in terms of the Kirchhoff stress tensor $\tau$}
\label{sec:special_case_1D_hypo_Kirchhoff_stress}
Taking into account the kinematics of the uniaxial case,
we can express the rate-formulation of the constitutive law also in terms of $\tau(\lambda) = \lambda \, \sigma(\lambda) = \DD_{\log \lambda}\widehat{\WW}(\log \lambda)$. This reads
\begin{align}
	\label{eqkirchrate}
	\frac{\DD}{\DD t}[\tau(\log \lambda)] &= \DD_{\log \lambda} \tau(\log \lambda) \, \frac{\dot{\lambda}}{\lambda} =: \H_{\tau}(\tau) \, D = \DD^2_{\log \lambda} \widehat{\WW}(\log \lambda) \, D \, , \quad \quad D = \frac{\dot{\lambda}}{\lambda} \, .
\end{align}
In this case, we observe therefore the purely Hessian-character of the corresponding induced tangent-stiffness $\H_{\tau}(\tau)$ in the sense that
\begin{align}
	\H_{\tau}(\tau) = \DD^2_{\log \lambda} \widehat{\WW}(\log \lambda) = \DD_{\log \lambda} \tau(\log \lambda) \quad \quad \text{(cf. \eqref{eq:derivative_tau})} \, .
\end{align}
Note the following relation between the Kirchhoff stress tensor $\tau(\lambda) = \widehat{\tau}(\log \lambda)$ and the Cauchy stress tensor $\sigma$ as well as the stiffness tensor $\H_{\tau}(\tau)$ and the \emph{logarithmic stiffness} $\H(\sigma)$
\begin{equation}
	\label{eq2.113001}
	\begin{alignedat}{2}
		\widehat{\tau}(\log \lambda) &= \mathrm{e}^{\log \lambda} \, \widehat{\sigma}(\log \lambda) = \lambda \, \sigma(\lambda) \\
		\implies \quad \H_{\tau}(\tau) &= \DD_{\log \lambda} \tau(\log \lambda) = \mathrm{e}^{\log \lambda} \, \widehat{\sigma}(\log \lambda) + \mathrm{e}^{\log \lambda} \, \DD_{\log \lambda} \widehat{\sigma}(\log \lambda) \\
		&= \lambda \, [\DD_{\log \lambda} \widehat{\sigma}(\log \lambda) + \widehat{\sigma}(\log \lambda)] = \lambda \, [\H(\widehat{\sigma}) + \widehat{\sigma}] = \lambda \, \left[\H\left(\frac{1}{\lambda} \, \tau \right) + \frac{1}{\lambda} \, \tau \right].
	\end{alignedat}
\end{equation}
This representation coincides with the three-dimensional rate formulation for the Kirchhoff stress, see \eqref{eq0.4}. \\
Assuming that $\H_{\tau}(\tau) \equiv 1$ (zero grade hypoelasticity), equation \eqref{eqkirchrate} integrates directly to the quadratic Hencky type energy, i.e.
\begin{align}
	1 \equiv \H_{\tau}(\tau) = \DD^2_{\log \lambda} \widehat{\WW}(\log \lambda) \quad \iff \quad \widehat{\WW}(\log \lambda) = \frac12 \log^2(\lambda).
\end{align}
Thus, $\H_{\tau} \equiv 1$ delivers $\widehat{\tau}(\widehat{\xi}) = \widehat{\xi} = \log \lambda$, so that the Kirchhoff stress $\tau$ is monotone in $\log \lambda$ (Hill's inequality), while the Cauchy stress $\sigma (\lambda)$ is now non-monotone. The Abaqus\textsuperscript{\texttrademark} and Ansys\textsuperscript{\texttrademark} "stability check" considers only the monotonicity of $\widehat{\tau}$ as function of $\log \lambda$. Thus, the Hencky energy is "stable" as seen by Abaqus\textsuperscript{\texttrademark} and Ansys\textsuperscript{\texttrademark}.
\begin{figure}[h!]
	\begin{center}
		\begin{minipage}[h!]{0.4\linewidth}
			\centering
			\includegraphics[scale=0.25]{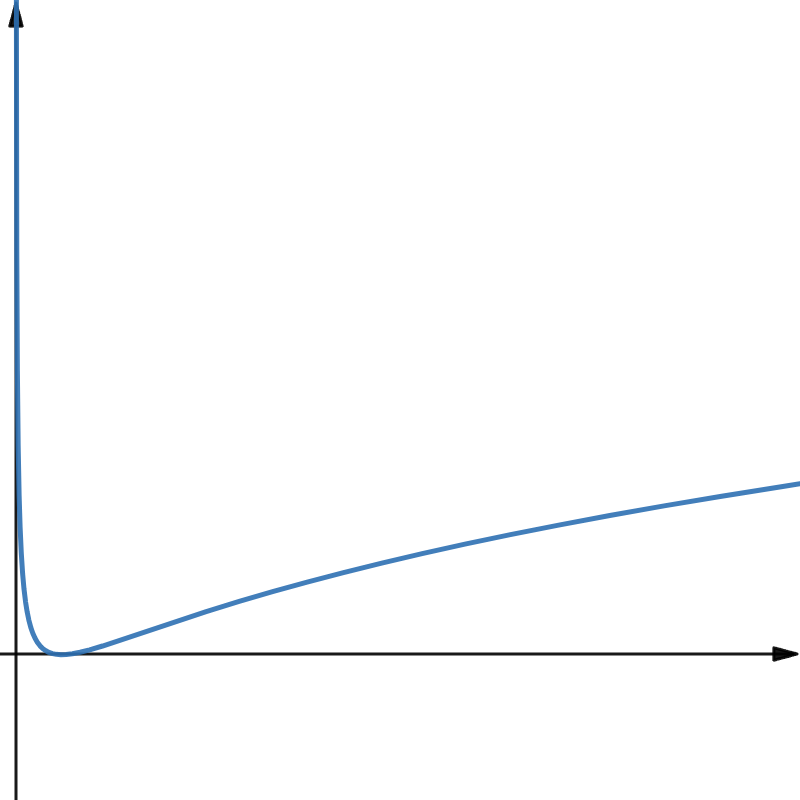}
			\put(-188,25){\footnotesize{$1$}}
			\put(-10,25){\footnotesize{$\lambda$}}
			\put(-225,190){\footnotesize{$\WW(\lambda)$}}
			\caption{Picture of the non-convex 1D-quadratic Hencky strain energy \break $W_{\text{Hencky}}(\lambda) = \frac{1}{2}\log^2(\lambda)$}
		\end{minipage}
		\qquad
		\begin{minipage}[h!]{0.4\linewidth}
			\includegraphics[scale=0.35]{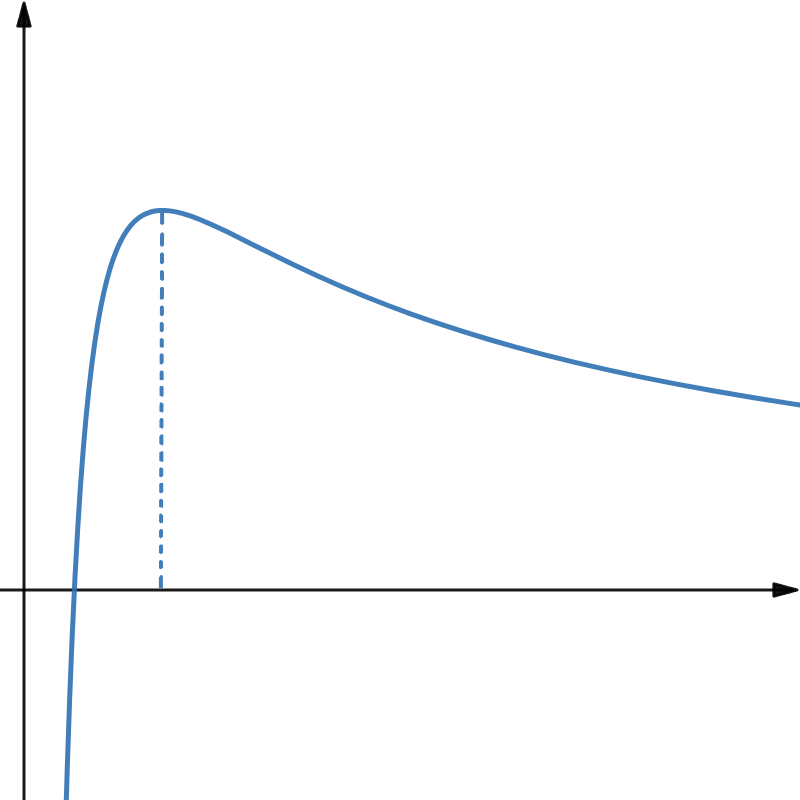}
			\put(-190,45){\footnotesize{$1$}}
			\put(-170,45){\small{$\mathrm{e}$}}
			\put(-5,40){\footnotesize{$\lambda$}}
			\put(-200,200){\footnotesize{$\sigma(\lambda)$}}
			\caption{Cauchy stress $\sigma(\lambda) = \frac{\log \lambda}{\lambda}$, non-monotone response for the Hencky energy $\WW(\lambda) = \frac12 \, \log^2(\lambda)$.}
		\end{minipage}
	\end{center}
\end{figure}
	\begin{figure}[h!]
	\begin{center}
		\begin{minipage}[h!]{0.4\linewidth}
			\includegraphics[scale=0.25]{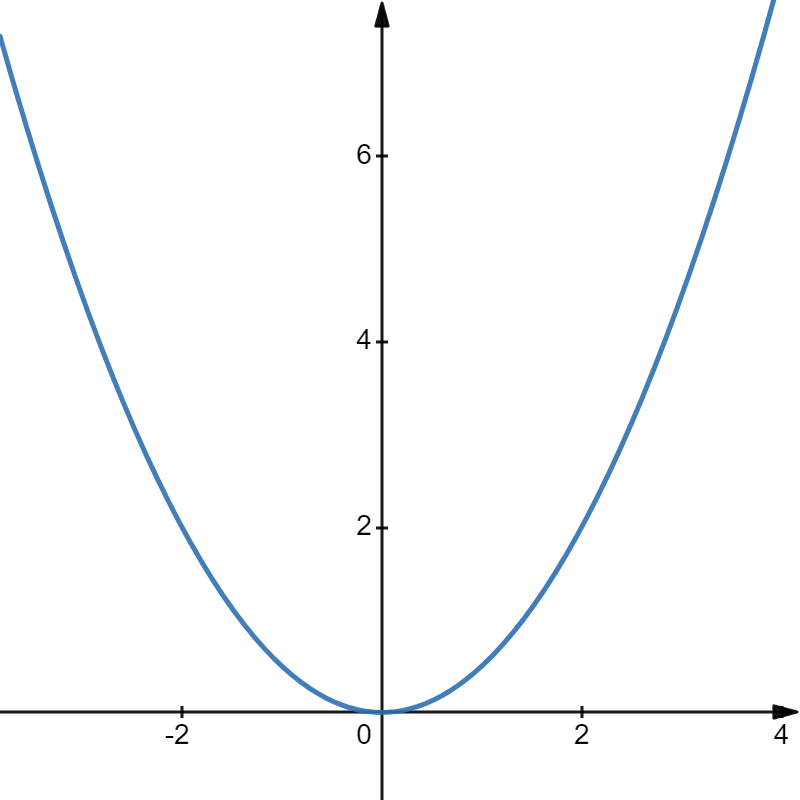}
			\put(-10,5){\footnotesize{$\log \lambda$}}
			\put(-145,190){\footnotesize{$\widehat{\WW}(\log \lambda)$}}
			\caption{Picture of the strain energy \break $\widehat{\WW}(\log \lambda) = \frac{1}{2}\log^2(\lambda)$, convex in $\log \lambda$.}
		\end{minipage}
		\qquad
		\begin{minipage}[h!]{0.4\linewidth}
			\includegraphics[scale=0.25]{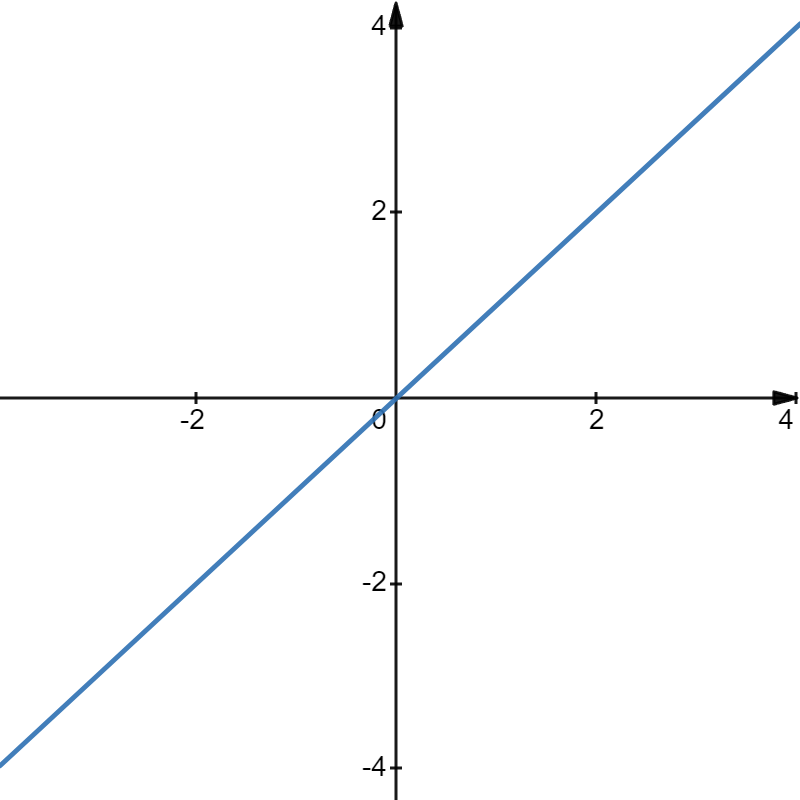}
			\put(-10,85){\footnotesize{$\log \lambda$}}
			\put(-140,190){\footnotesize{$\widehat{\tau}(\log \lambda)$}}
			\caption{Picture of the monotone Kirchhoff stress $\widehat{\tau}(\log \lambda) = \DD_{\log \lambda}\widehat{\WW}(\log \lambda)$. The monotonicity of $\widehat{\tau}$ as a function of $\log \lambda$ corresponds to the satisfaction of \emph{Hill's inequality} \eqref{eq:Hills_inequality} for the Hencky model.}
		\end{minipage}
	\end{center}
\end{figure}
\begin{figure}[h!]
	\begin{center}
		\begin{minipage}[h!]{0.4\linewidth}
			\includegraphics[scale=0.25]{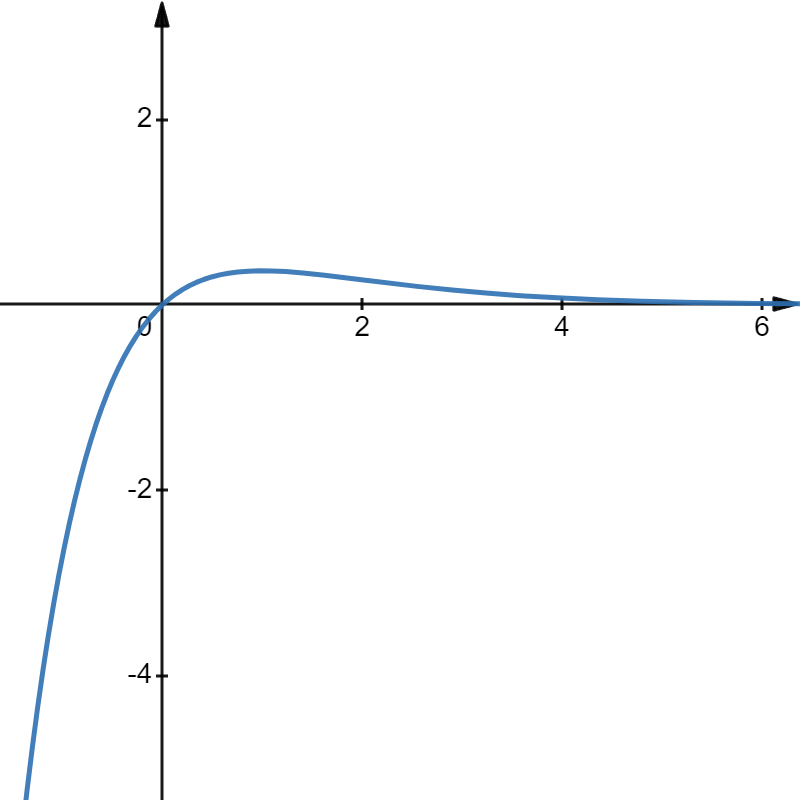}
			\put(-15,108){\footnotesize{$\log \lambda$}}
			\put(-150,190){\footnotesize{$\widehat{\sigma}(\log \lambda)$}}
			\caption{Non-monotone Cauchy stress \break $\widehat{\sigma}(\log \lambda) = \log(\lambda) \, \mathrm{e}^{-\log (\lambda)}$ in terms of the logarithmic strain $\log \lambda$.}
		\end{minipage}
		\qquad
		\begin{minipage}[h!]{0.4\linewidth}
			\centering
			\includegraphics[scale=0.25]{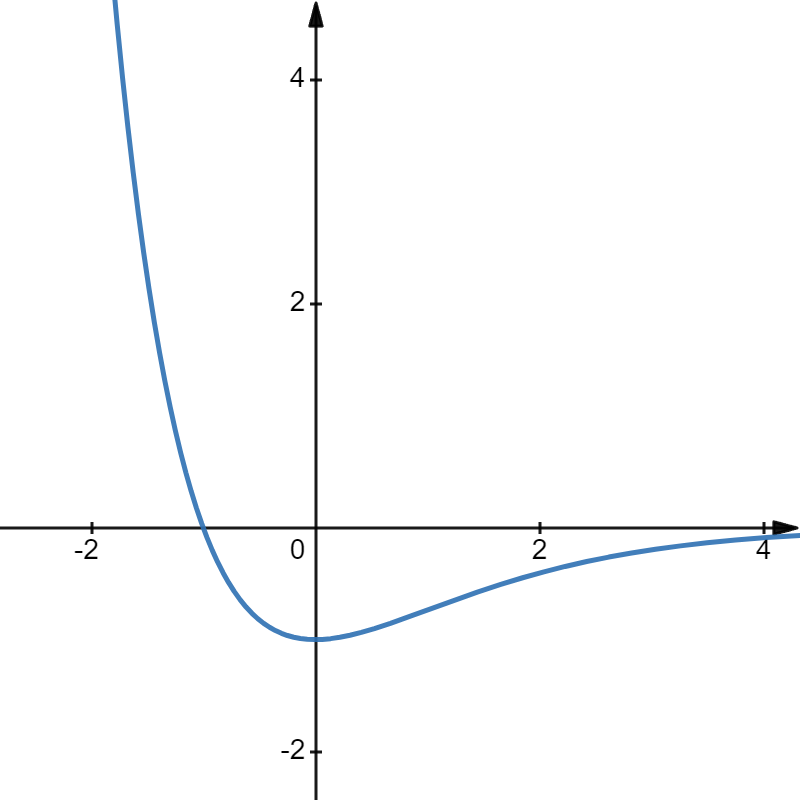}
			\put(-15,75){\footnotesize{$\log \lambda$}}
			\put(-110,190){\footnotesize{$\H(\log \lambda)$}}
			\caption{Picture of the non-positive logarithmic stiffness $\H(\log \lambda) \coloneqq \DD_{\log \lambda} \, \widehat{\sigma} (\log \lambda)$ given by $\H(\log \lambda) = -(\log \lambda + 1) \, \mathrm{e}^{-\log \lambda}$.}
		\end{minipage}
	\end{center}
\end{figure}

\subsection{The Zaremba-Jaumann rate seen as a Lie derivative}
\label{appendix:Jaumann_and_Lie}
The Lie derivative of a tensor is a derivation performed by means of pull-back and push-forward operations with respect to the \emph{flow operator} of a vector field \cite{Hughes1977a,Marsden1983a,Federico2022a}. In continuum mechanics, this vector field is the (Eulerian) velocity $v$ of the motion $\varphi$, and the Lie derivative of a tensor can be written in terms of regular push-forward and pull-back operations, as \emph{the push-forward of the time derivative of the pull-back} of the tensor \cite{Bonet2008a,Federico2022a}.

For the case of the Kirchhoff stress (which, strictly speaking, is a tensor with \emph{contravariant} components $\tau^{ij}$), the Lie derivative reads
\begin{align}
	\label{eq:Lie_derivative_tau}
	\mathcal{L}_v \tau =
	F \left[ \frac{\DD}{\DD t} [F^{-1} \tau \, F^{-T}] \right] F^T =
	\frac{\DD}{\DD t} [\tau] - L \, \tau - \tau \, L^T,
\end{align}
where $F^{-1} \tau \, F^{-T} = S_2$ is the second Piola-Kirchhoff stress, $\frac{\DD}{\DD t}[\tau]$, with components
	\begin{align}
		\label{eq:substantial_derivative_tau}
		\left[ \frac{\DD}{\DD t} [\tau] \right]^{ij} = \tau^{ij}{}_{|k} \, v_k + \frac{\partial \tau^{ij}}{\partial t} \, ,
	\end{align}
is the standard material (substantial) derivative of $\tau$ (where $\tau^{ij}{}_{|k} = \partial \tau^{ij} / \partial x^k + \gamma^i_{pk} \, \tau^{pj} + \gamma^j_{qk} \, \tau^{iq}$ is the covariant derivative with Christoffel symbols $\gamma^i_{jk}$ of the Levi-Civita connection associated with the metric tensor $g$ \cite{Marsden1983a}) and $L = \dot{F} \, F^{-1}$ arises from the derivative
\begin{align}
	\frac{\DD}{\DD t} [F^{-1}] = - F^{-1} L.
\end{align}
Hughes and Marsden~\cite{Hughes1977a} and Marsden and Hughes~\cite{Marsden1983a} elegantly showed how the Zaremba-Jaumann rate, which in covariant formalism reads
\begin{align}
	\label{eq:Jaumann_rate_tau}
	\frac{\DD^{\ZJ}}{\DD t}[\tau] = \frac{\DD}{\DD t}[\tau] + \tau \, W \, g^{-1} - g^{-1} \, W \, \tau \, ,
\end{align}
is obtained as the linear combination of Lie derivatives
	\begin{align}
		\label{eq:Jaumann_rate_tau_w_Lie}
		\frac{\DD^{\ZJ}}{\DD t}[\tau] =
		\frac{1}{2} \Big[ [\mathcal{L}_v(\tau \, g)] \, g^{-1} + g^{-1} \, [\mathcal{L}_v(g \, \tau)] \Big],
	\end{align}
	where $g$ is the metric tensor, with covariant components $g_{ij}$, and $g^{-1}$ is its inverse, with contravariant components $g^{ij}$. To show the equivalence of~\eqref{eq:Jaumann_rate_tau} and~\eqref{eq:Jaumann_rate_tau_w_Lie}, we must temporarily switch to a covariant formalism, and consider that $\tau$ has contravariant components $\tau^{ij}$, so that $\tau \, g$ is the tensor with components $(\tau \, g)^i{}_k = \tau^{ij} \, g_{jk}$ and $g \, \tau$ is the tensor with components $(g \, \tau)_i{}^k = g_{ij} \, \tau^{jk}$ and, in components, Eq.~\eqref{eq:Jaumann_rate_tau_w_Lie} reads
	\begin{align}
		\label{eq:Jaumann_rate_tau_w_Lie_comp}
		\left[ \frac{\DD^{\ZJ}}{\DD t}[\tau] \right]^{il} =
		\frac{1}{2} \Big[ [\mathcal{L}_v(\tau \, g)]^i{}_k \, g^{kl} + g^{ik} \, [\mathcal{L}_v(g \, \tau)]_k{}^l \Big].
	\end{align}
	We also need to consider two other relations. First, in covariant formalism, the velocity gradient $L$ is decomposed into
	\begin{align}
		\label{eq:covariant_decomposition_velocity_gradient}
		L = g^{-1} D + g^{-1} W =
		g^{-1} \frac{1}{2}(g \, L + L^T g) + g^{-1} \frac{1}{2}(g \, L - L^T g)
	\end{align}
	so that the rate of deformation $D$ and the vorticity tensor $W$ both have covariant components. Second, the Lie derivative satisfies Leibniz' rule in a strict sense, i.e., for any two spatial tensors $\mathbbs{A}$ and $\mathbbs{B}$ of \emph{any} order, we have~\cite{Marsden1983a}
	\begin{align}
		\label{eq:Lie_derivative_Leibniz}
		\mathcal{L}_v(\mathbbs{A} \otimes \mathbbs{B}) = (\mathcal{L}_v\,\mathbbs{A}) \otimes \mathbbs{B} + \mathbbs{A} \otimes (\mathcal{L}_v\,\mathbbs{B}).
	\end{align}
	Since the Leibniz rule~\eqref{eq:Lie_derivative_Leibniz} is valid for the tensor product of $\mathbbs{A}$ and $\mathbbs{B}$, it is also valid for any contraction of $\mathbbs{A}$ and $\mathbbs{B}$. Therefore, since the Lie derivative of the metric $g$ is twice the rate of deformation $D$ \cite{Marsden1983a}, i.e.,
	\begin{align}
		\label{eq:Lie_derivative_metric}
		\mathcal{L}_v\,g =
		F^{-T} \left[ \frac{\DD}{\DD t} [F^T g \, F] \right] F^{-1} =
		F^{-T} \left[ F^T L^T g \, F + F^T g \, L \, F \right] F^{-1} =
		L^T g  + g \, L =
		2 \, D,
	\end{align}
	the Lie derivatives featuring in~\eqref{eq:Jaumann_rate_tau_w_Lie} can be written
	\begin{equation}
		\begin{alignedat}{2}
			[\mathcal{L}_v(\tau \, g)] \, g^{-1}
			& = [(\mathcal{L}_v\,\tau) \, g + \tau \, (\mathcal{L}_v\,g)] \, g^{-1}
			= \mathcal{L}_v\,\tau + \tau \, (2 \, D) \, g^{-1} \, ,
			\\
			g^{-1} \, [\mathcal{L}_v(g \, \tau)]
			& = g^{-1} \, [(\mathcal{L}_v\,g) \, \tau + g \, (\mathcal{L}_v\,\tau)]
			= g^{-1} \, (2 \, D) \, \tau + \mathcal{L}_v\,\tau.
		\end{alignedat}
	\end{equation}
	Substituting this result into~\eqref{eq:Jaumann_rate_tau_w_Lie}, we obtain
	\begin{align}
		\label{eq:Jaumann_rate_and_Lie_derivative_tau_covariant}
		\frac{\DD^{\ZJ}}{\DD t}[\tau] =
		\mathcal{L}_v\,\tau + \tau \, D \, g^{-1} + g^{-1} D \, \tau.
	\end{align}
	Now, we substitute the expression~\eqref{eq:Lie_derivative_tau} of the Lie derivative of $\tau$, which gives
	\begin{align}
		\label{eq:Jaumann_tau_and_Lie_derivative_tau}
		\frac{\DD^{\ZJ}}{\DD t}[\tau] =
		\frac{\DD}{\DD t} [\tau] - L \, \tau - \tau \, L^T + \tau \, D \, g^{-1} + g^{-1} D \, \tau
	\end{align}
	Substitution of the symmetric-skew decomposition~\eqref{eq:covariant_decomposition_velocity_gradient} of the velocity gradient $L$ into~\eqref{eq:Jaumann_tau_and_Lie_derivative_tau} and some simple manipulation yield the expression~\eqref{eq:Jaumann_rate_tau} of the Zaremba-Jaumann rate of $\tau$.

	The Cauchy stress $\sigma$ is a tensor of the same type as the Kirchhoff stress $\tau$ (with contravariant components $\sigma^{ij}$). Therefore, the Lie derivative of the Cauchy stress $\sigma$ has the same form seen in~\eqref{eq:Lie_derivative_tau} for the case of the Kirchhoff stress $\tau$, i.e.,
\begin{align}
	\label{eq:Lie_derivative_sigma}
	\mathcal{L}_v \sigma =
	F \left[ \frac{\DD}{\DD t} [F^{-1} \sigma \, F^{-T}] \right] F^T =
	\frac{\DD}{\DD t} [\sigma] - L \, \sigma - \sigma \, L^T \, .
\end{align}
Similarly, the Zaremba-Jaumann rate has expressions analogous to those seen in Eqs.~\eqref{eq:Jaumann_rate_tau} and~\eqref{eq:Jaumann_rate_and_Lie_derivative_tau_covariant} for $\tau$. In terms of the substantial derivative $(\DD/\DD t)[\sigma]$ and the spin tensor $W$, we have
\begin{align}
	\label{eq:Jaumann_rate_sigma}
	\frac{\DD^{\ZJ}}{\DD t}[\sigma] = \frac{\DD}{\DD t}[\sigma] + \sigma \, W \, g^{-1} - g^{-1} \, W \, \sigma \, ,
\end{align}
and, in terms of the Lie derivative $\mathcal{L}_v\,\sigma$ and the deformation rate $D$, we have
\begin{align}
	\label{eq:Jaumann_rate_and_Lie_derivative_sigma_covariant}
	\frac{\DD^{\ZJ}}{\DD t}[\sigma] =
	\mathcal{L}_v\,\sigma + \sigma \, D \, g^{-1} + g^{-1} D \, \sigma \, .
\end{align}
However, for the purpose of calculating the elasticity tensor related to the Zaremba-Jaumann rate~\eqref{eq:Jaumann_rate_sigma} of $\sigma$ it is more convenient to express it in terms of the \emph{Truesdell rate} rather than, as we had done with the Kirchhoff stress $\tau$, in terms of the Lie derivative. The Truesdell rate is, however, closely related to the Lie derivative \cite{Marsden1983a,Federico2022a}. Indeed, exploiting the Leibniz' rule~\eqref{eq:Lie_derivative_Leibniz} and the fact that the Lie derivative of a scalar is simply the substantial time derivative (here, we consider $J$ as an \emph{Eulerian} scalar field, i.e., a function of the spatial point $\xi$), we have
\begin{align}
	\label{eq:Truesdell_rate_sigma}
	\frac{\DD^{\TR}}{\DD t}[\sigma] =
	J^{-1} \mathcal{L}_v (J \, \sigma) =
	J^{-1} \left[ J \, (\mathcal{L}_v \sigma) + (\mathcal{L}_v J) \, \sigma) \right] = 
	\mathcal{L}_v \sigma + \tr(D) \, \sigma \, .
\end{align}
where we used $\dot{J} = J \, \mathrm{div}\,v = J \, \tr(D)$. Moreover, considering the fact that the Lie derivative is the push-forward of the time derivative of the pull-back, the expression~\eqref{eq:Truesdell_rate_sigma} implies that the Truesdell rate is \emph{the forward Piola transform of the time derivative of the backward Piola transform} \cite{Federico2022a}, i.e.,
\begin{align}
	\label{eq:Truesdell_rate_sigma_Piola_transforms}
	\frac{\DD^{\TR}}{\DD t}[\sigma] =
	J^{-1} \mathcal{L}_v (J \, \sigma) =
	J^{-1} F \left[ \frac{\DD}{\DD t} [J \, F^{-1} \sigma \, F^{-T}] \right] F^T =
	\frac{\DD}{\DD t} [\sigma] - L \, \sigma - \sigma \, L^T + \tr(D) \, \sigma \, ,
\end{align}
where $J \, F^{-1} \sigma \, F^{-T} = S_2$ is the second Piola-Kirchhoff stress. Now, solving for the Lie derivative $\mathcal{L}_v \sigma$ in~\eqref{eq:Truesdell_rate_sigma} and substituting in~\eqref{eq:Jaumann_rate_and_Lie_derivative_sigma_covariant}, we obtain the expression
\begin{align}
	\label{eq:Jaumann_rate_and_Truesdell_rate_sigma_covariant}
	\frac{\DD^{\ZJ}}{\DD t}[\sigma] =
	\frac{\DD^{\TR}}{\DD t}[\sigma] + \sigma \, D \, g^{-1} + g^{-1} D \, \sigma - \tr(D) \, \sigma \, .
\end{align}

	Reverting to Cartesian coordinates, in which the metric tensor $g$ and its inverse $g^{-1}$ behave like the identity tensor, the relation~\eqref{eq:Jaumann_rate_and_Lie_derivative_tau_covariant} between the Zaremba-Jaumann rate~\eqref{eq:Jaumann_rate_tau} of $\tau$ and the Lie derivative~\eqref{eq:Lie_derivative_tau} of $\tau$ reads
\begin{align}
	\label{eq:Jaumann_rate_and_Lie_derivative_tau}
	\frac{\DD^{\ZJ}}{\DD t}[\tau] = \mathcal{L}_v \tau + \tau \, D + D \, \tau \, ,
\end{align}
and the relation~\eqref{eq:Jaumann_rate_and_Truesdell_rate_sigma_covariant} between the Zaremba-Jaumann rate~\eqref{eq:Jaumann_rate_sigma} of $\sigma$ and the Truesdell rate~\eqref{eq:Truesdell_rate_sigma} of $\sigma$ reads
\begin{align}
	\label{eq:Jaumann_rate_and_Truesdell_rate_sigma}
	\frac{\DD^{\ZJ}}{\DD t}[\sigma] = \frac{\DD^{\TR}}{\DD t}[\sigma] + \sigma \, D + D \, \sigma \, - \tr(D) \, \sigma.
\end{align}

\end{appendix}

\end{document}